\documentclass{amsart} 

\usepackage{amsmath,amssymb,mathrsfs,amsthm,delarray}
\usepackage[all,cmtip]{xy}
\numberwithin{equation}{section}

\newtheorem{Theorem}{Theorem}
\newtheorem*{maintheorem}{Main Theorem}
\newtheorem*{theorem3'}{Theorem 3'}
\newtheorem{Proposition}{Proposition}[section]
\newtheorem{Lemma}{Lemma}[section]
\newtheorem{definition}{Definition}[section]

\renewcommand{\v}{\mathbf{v}}

\newcommand{\C}{\mathbb C}
\newcommand{\CC}{\widehat{\mathbb{C}}}
\newcommand{\Q}{\mathbb{Q}}
\newcommand{\R}{\mathbb{R}}
\newcommand{\Z}{\mathbb{Z}}
\newcommand{\Poly}{\text{Poly}}
\usepackage{geometry} 
\usepackage{graphicx} 

\title{Primitive tuning via quasiconformal surgery}
\author{Weixiao Shen and Yimin Wang}
\date{\today}
\address{Shanghai Center for Mathematical Sciences,
Fudan University, No. 2005 Songhu Road, Shanghai 200438, China}
\email{wxshen@fudan.edu.cn, 16110840001@fudan.edu.cn}
\begin{document}
\maketitle

\begin{abstract} Using quasiconformal surgery, we prove that any primitive, postcritically-finite hyperbolic polynomial can be tuned with an arbitrary generalized polynomial with non-escaping critical points, generalizing a result of Douady-Hubbard for quadratic polynomials to the case of higher degree polynomials. This solves affirmatively a conjecture by Inou and Kiwi on surjectivity of the renormalization operator on higher degree polynomials in one complex variable.
\end{abstract}
\section{Introduction}
Quasiconformal surgery is a powerful technique in the theory of holomorphic dynamics of one variable. The process often consists of two steps. First, we construct a quasi-regular map with certain dynamical properties, by modifying part of a holomorphic dynamical system, or extending the definition of a partially defined holomorphic dynamical system. Then we prove existence of a rational map which is (essentially) conjugate to the quasi-regular mapping. Successful applications of this technique include Douady-Hubbard's theory of polynomial-like mappings (\cite{DH2}) and Shishikura's sharp bounds on the number of Fatou cycles (\cite{Shishi}), among others.

Tuning is an operator introduced by Douady-Hubbard~\cite{DH2} to prove existence of small copies in the Mandelbrot set. Recall that the Mandelbort set consists of all complex numbers $c$ for which $P_c(z)=z^2+c$ has a connected Julia set. Let $P_a(z)=z^2+a$ be a quadratic polynomial for which the critical point $0$ has period $p$.  Douady-Hubbard proved that there is a homeomorphism $\tau=\tau_a$, from $\mathcal{M}$ into itself with $\tau_a(0)=a$ and with the property that the $p$-th iterate of $P_{\tau(c)}$ has a quadratic-like restriction which is hybrid equivalent to $P_c$. Intuitively, the filled Julia set of $P_{\tau(c)}$ is obtained from that of $P_a$ by replacing each bounded Fatou component of $P_a$ with a copy of the filled Julia set of $P_c$.
Douady-Hubbard's argument involved detailed knowledge of the combinatorics of the Mandelbrot set and also continuity of the straightening map in the quadratic case, and thus breaks down for higher degree polynomials.

In~\cite{IK}, Inou-Kiwi defined a natural analogy of the map $\chi=\tau^{-1}$ for higher degree polynomials, which we will call the {\em IK straightening map}. Given a postcritically finite hyperbolic polynomial $f_0$ of degree $d\ge 2$, together with an internal angle system, they defined a map $\chi$ from a certain space $\mathcal{R}(f_0)$ into a space $\mathcal{C}(T)$ which consists of generalized polynomials over the reduced mapping scheme $T=T(f_0)$ of $f_0$ with fiberwise connected Julia set. The map $\chi$ is injective and in general not continuous~\cite{I3}.  The map $f\in \mathcal{R}(f_0)$ is the tuning of $f_0$ with $\chi(f)$ in the sense of Douady and Hubbard when the Julia set $\chi(f)$ is fiberwise locally connected. In the case that $f_0$ is {\em primitive}, i.e. the bounded Fatou components have pairwise disjoint closures, the set $\mathcal{R}(f_0)$ coincide with a combinatorially defined set $\mathcal{C}(f_0)$
which is known to be compact. Inou-Kiwi's argument is mostly combinatorial and they posed a conjecture that in the case that $f_0$ is primitive, $\chi$ is surjective and with connected domain. The fact that $\chi$ is surjective means that $f_0$ can be tuned by every $g\in \mathcal{C}(T)$.

In this paper, we shall prove the conjecture of Inou and Kiwi, using quasiconformal surgery technique. In particular, this shows that a primitive postcritically finite hyperbolic map $f_0$ can be tuned by each $g\in \mathcal{C}(T(f_0))$.

\begin{maintheorem} Let $f_0$ be a postcritically finite hyperbolic polynomial that is primitive and fix an internal angle system for $f_0$. Then the IK straightening map $\chi:\mathcal{C}(f_0)\to \mathcal{C}(T)$ is bijective and $\mathcal{C}(f_0)$ is connected.
\end{maintheorem}
We shall recall the definition of Inou-Kiwi's straightening map later. Let us just mention now that if $f_0(z)=z^d+a$ for some $a\in \C$ then $\mathcal{C}(T)$ is the set of all monic centered polynomials of degree $d$ which have connected Julia sets.

The main part of the proof is to show the surjectivity of the map $\chi$. It is fairly easy to construct a quasi-regular map $\widetilde{f}$ which has a generalized polynomial-like restriction hybrid equivalent to a given generalized polynomial in $\mathcal{C}(T)$, via qc surgery in a union of annuli domains around the critical Fatou domains of $f_0$. In order to show that there is a polynomial map with essentially the same dynamics as $\widetilde{f}$, we run Thurston's algorithm to $\widetilde{f}$. We modify an argument of Rivera-Letelier~\cite{R-L} to show the convergence. In order to control distortion, we use a result of Kahn~\cite{Kahn} on removability of certain Julia sets. After surjectivity is proved, we deduce connectivity of $\mathcal{C}(f_0)$ from that of $\mathcal{C}(T)$ which is a theorem of Branner-Hubbard~\cite{BH} and Lavaurs~\cite{La}.

In \cite{EY}, qc surgery was successfully applied to construct intertwining tuning. Our case is more complicated since the surgery involves the non wandering set of the dynamics.

The paper is organized as follows. In \S\ref{sec:IK}, we recall definition of generalized polynomial-like maps and the IK straightening map. In \S\ref{sec:puzzle}, we construct a specific Yoccoz puzzle for postcritically finite hyperbolic primitive polynomials which is used to construct renormalizations. In \S\ref{sec:Kahn}, we prove a technical distortion lemma. In \S\ref{sec:Thurston}, we prove a convergence theorem for Thurston algorithm. The proof of the main theorem is given in \S\ref{sec:surgery} using qc surgery.

\medskip
\noindent
{\bf Acknowledgment.} We would like to thank the referee for carefully reading the paper and in particular, for pointing out an error in Section 3 of the first version. This work was supported by NSFC No. 11731003.
\section{The IK straightening map}\label{sec:IK}
We recall some basic notations and the definition of IK straightening maps. The multi-critical nature of the problem makes the definition a bit complicated.

Let $\text{Poly}_d$ denote the set of all monic centered polynomials of degree $d$, i.e. polynomials of the form $z\mapsto z^d+ a_{d-2}z^{d-2}+\cdots+a_0$, with $a_0,a_1,\cdots, a_{d-2}\in \C$. For each $f\in \text{Poly}_d$, let $K(f)$ and $J(f)$ denote the filled Julia set and the Julia set respectively.
Let $P(f)$ denote the postcritical set of $f$:
$$P(f)=\overline{\bigcup_{f'(c)=0, c\in \C}\bigcup_{n=1}^\infty \{f^n(c)\}}.$$

Let $$\mathcal{C}_d=\{f\in \Poly_d: K(f) \text{ is connected}\}.$$
\subsection{External Rays and Equipotential Curves}
For $f\in Poly_d$, the {\em Green function} is defined as
$$G_f(z)=\lim_{n\to\infty} \frac{1}{d^n}\log^+ |f^n(z)|,$$
where $\log^+=\max(\log , 0).$
The Green function is continuous and subharmonic in $\C$, satisfying $G_f(f(z))=dG_f(z)$. It is nonnegative, harmonic and takes positive values exactly in the attracting basin of $\infty$
\[B_f(\infty):=\{z\in \mathbb C\mid f^n(z) \to \infty\ as~n \to \infty\}=\C\setminus K(f).\]

Assume $f\in\mathcal{C}_d$. Then there exists a unique conformal map \[\phi_f:B_f(\infty) \to \{z\mid |z|>1\}\]
such that $\phi(z)/z\to 1$ as $z\to\infty$ and such that $\phi_f\circ f(z)=(\phi_f(z))^d$ on $B_f(\infty)$.
The {\em external ray of angle $t \in \R/\Z$} is defined as
\[\mathcal R_{f}(t)=\{\phi_f^{-1}(re^{i2\pi t})\mid 1<r<\infty\},\]
and the {\em equipotential curve of pontential $l\in (0,\infty)$} as
\[E_{f}(l)=\{\phi_f^{-1}(e^{l+i2\pi\theta})\mid 0\le\theta<1\}.\]

We say the external ray $\mathcal R_f(t)$ {\em lands} at some point $z_0$ if $\lim_{r\to 1}\phi_f^{-1}(re^{i2\pi t})=z_0$. It is known that for each $t\in \mathbb{Q}/\Z$, $\mathcal{R}_f(t)$ lands at some point. On the other hand, if $z_0$ is a repelling or parabolic periodic point, then there exists at least one but at most finitely many external rays landing at $z_0$. See for example~\cite{Mil1}.

The {\em rational lamination} of $f$, denoted by $\lambda(f)$, is the equivalence relation on $\Q/\Z$ so that $\theta_1\sim \theta_2$ if and only if $\mathcal{R}_f(\theta_1)$ and $\mathcal{R}_f(\theta_2)$ land at the same point.

\subsection{Generalized polynomials over a scheme}\label{subsec:GPL}
Now let us consider a postcritically finite hyperbolic polynomial $f_0$.
Following Milnor \cite{Mil3}, we define the {\em reduced mapping scheme} $$T(=T(f_0))=(|T|,\sigma,\delta)$$ of $f_0$ as follows. Let $|T|$ denote the collection of critical bounded Fatou components of $f_0$. For each $\v\in |T|$, let $r(\v)$ be the minimal positive integer such that $f_0^{r(\v)}(\v)$ is again a critical Fatou component. Define $\sigma:|T|\to |T|$ and $\delta: |T|\to \{2,3,\cdots\}$ by
$$\sigma(\v)=f_0^{r(\v)}(\v)\mbox{ and }\delta(\v)=\text{deg}(f_0^{r(\v)}|\v)=\text{deg}(f_0|\v).$$

\begin{definition}[Generalized Polynomials] A generalized polynomial $g$ over $T$ is a map
\[g:|T|\times \mathbb C \to |T|\times \mathbb C\]
such that $g(\v,z)=(\sigma(\v),g_{\v}(z))$ where $g_{\v}(z)$ is a monic centered polynomial of degree $\delta(\v)$.
\end{definition}

Denote by $P(T)$ the set of all generalized polynomials over the scheme $T$.
Given $g \in P(T)$, the {\em filled Julia set} $K(g)$ of $g$ is the set of points in $|T|\times \mathbb C$ whose forward orbits are precompact. The boundary $\partial K(g)$ of the filled Julia set is called the {\em Julia set} $J(g)$ of $g$. The filled Julia set $K(g)$ is called {\em fiberwise connected} if $K(g,\v):=K(g)\cap (\{\v\}\times \mathbb C)$ is connected for each $\v\in |T|$. Let $\mathcal C(T)$ be the set of all the generalized polynomials with fiberwise connected filled Julia set over $T$.

We shall need external rays and Green function for $g\in \mathcal{C}(T)$ which can be defined similarly as in the case of a single polynomial. Indeed, for each $\v\in |T|$ there exists a unique conformal $\varphi_{\v,g}:\C\setminus K(g,\v)\to \C\setminus \overline{\mathbb{D}}$ such that $\varphi_{\v,g}(z)/z\to 1$ as $z\to\infty$ and these maps satisfy $\varphi_{\sigma(\v),g}(g_\v(z))=\varphi_{\v,g}(z)^{\delta(\v)}$. For $t\in\mathbb{R}/\mathbb{Z}$, the {\em external ray}
$\mathcal{R}_g(\v,t)$ is defined as $\varphi_{\v,g}^{-1}(\{re^{2\pi i t}: r>1\})$. The {\em Green function} of $g$ is defined as
$$G_g(\v,z)=\left\{\begin{array}{ll}
\log |\varphi_\v(z)|, & \mbox{ if } z\not\in K(g,\v);\\
0 &\mbox{ otherwise.}
\end{array}
\right.
$$

\subsection{Generalized polynomial-like maps}
We shall now recall the definition of generalized polynomial-like map over the mapping scheme $T$. We say $U$ is a {\em topological multi-disk} in $|T|\times \mathbb C$ if there exist topological disks $\{U_\v\}_{\v\in|T|}$ in $\C$ such that $U \cap (\{\v\}\times \mathbb C)=\{\v\}\times U_\v$.
\begin{definition}
An AGPL (almost generalized polynomial-like) map $g$ over the scheme $T$ is a map
\[g:U \to U', (\v,z)\mapsto (\sigma(\v),g_{\v}(z)),\]
with the following properties:
\begin{itemize}
\item $U, U'$ are two topological multi-disks in $|T|\times \mathbb C$ and $U_\v\subsetneq U'_\v$ for each $\v\in |T|$;
\item $g_{\v}: U_\v \to U'_{\sigma(\v)}$ is a proper map of degree $\delta(\v)$ for each $\v$;
\item The set $K(g):=\bigcap\limits_{n=0}^{\infty} g^{-n}(U)$, called the {\em filled-Julia set of $g$}, is compactly contained in $U$.
\end{itemize}
If in addition $U\Subset U'$, then we say that $g$ is a GPL (generalized polynomial-like) map.
\end{definition}

It should be noted that every AGPL map has a GPL restriction with the same filled-Julia set. See \cite[Lemma 2.4]{LY}.

Let $g_1, g_2$ be two AGPL maps over $T$. We say that they are {\em qc conjugate} if there is a fiberwise qc map $\varphi$ from a neighborhood of $K(g_1)$ onto a neighborhood of $K(g_2)$, sending the $\v$ fiber of $g_1$ to the $\v$ fiber of $g_2$, such that $\varphi\circ g_1=g_2\circ \varphi$ near $K(g_1)$. We say that they are {\em hybrid equivalent} if they are qc conjugate and we can choose $\varphi$ such that $\bar{\partial} \varphi=0$ a.e. on $K(g_1)$.

The Douady-Hubbard straightening theorem~\cite{DH2} extends in a straightforward way: every AGPL map $g$ is hybrid equivalent to a generalized polynomial $G$. In the case that the filled Julia set of $g$ is fiberwise connected, $G$ is determined up to affine conjugacy. For monic centered quadratic polynomials, each affine conjugacy class consists of a single polynomial. For more general maps, it is convenient to introduce an (external) marking to uniquely determine $G$ for a given $g$.

\begin{definition}[Access and external marking]Let $f: U \to U'$ be an AGPL map over the mapping scheme
$T$ with fiberwise connected filled Julia set. A path to $K(f)$ is a continuous map
$\gamma: [0, 1] \to U'$ such that $\gamma ((0, 1]) \subset U' \backslash K(f)$ and $\gamma(0) \in J(f)$. We say two paths $\gamma_0$ and $\gamma_1$
to $K(f)$ are homotopic if there exists a continuous map $\tilde \gamma : [0, 1] \times [0, 1] \to U$ such that
\begin{enumerate}
\item $t \mapsto \tilde \gamma (s, t)$ is a path to $K(f)$ for all $s \in [0, 1]$;
\item $\tilde \gamma (0, t) = \gamma_0(t)$ and $\tilde\gamma (1, t) = \gamma_1(t)$ for all $t \in [0, 1]$;
\item $\tilde \gamma (s, 0) = \gamma_0(0)$ for all $s \in [0, 1]$.
\end{enumerate}
An access to $K(f)$ is a homotopy class of paths to $K(f)$.\par
An external marking of $f$ is a collection $\Gamma=(\Gamma_{\v})_{\v \in |T|}$ where each $\Gamma_{\v}$ is an access to $K(f)$, contained in $\{\v\}\times \C$, such that $\Gamma$ is forward
invariant in the following sense. For every $\v \in |T|$ and every representative $\gamma_{\v} \subset (\{\v\}\times\C) \cap U$ of $\Gamma_{\v}$, the connected component of $f(\gamma_{\v})\cap U$ which intersects $J(f)$ is a representative of $\Gamma_{\sigma(\v)}$.
\end{definition}
For a generalized polynomial $g \in \mathcal C(T)$, there is a {\em standard external marking of $g$}, given by the external rays $(\mathcal R(\v,0))_{\v \in |T|}$ with angle $0$.

\begin{Theorem}[Straightening] Let $g$ be an AGPL map over $T$ with fiberwise connected filled Julia set and let $\Gamma$ be an external marking of $g$. There is a unique $f\in \mathcal{C}(T)$ such that there is a hybrid conjugacy between $g$ and $f$ which sends the external marking $\Gamma$ to the standard marking of $f$.
\end{Theorem}

See \cite[Theorem A]{IK} for a proof.

\subsection{The Inou-Kiwi map}
Let $f_0, T$ and $r:|T|\to\mathbb{N}$ be as in \S\ref{subsec:GPL} and assume $f_0$ is primitive. It is well-known that $\partial \v$ is a Jordan curve for each $\v\in |T|$.
Let $$\mathcal{C}(f_0)=\{f\in \text{Poly}_d: \lambda(f)\supset \lambda(f_0)\},$$
where $\lambda(f)$ and $\lambda(f_0)$ are the rational laminations of $f$ and $f_0$ respectively.

For each $f\in \mathcal{C}(f_0)$, Inou-Kiwi constructed an AGPL (in fact, GPL) map
\begin{equation}\label{eqn:lambda0R}
F:\bigcup_{\v\in |T|} \{\v\}\times U_\v\to \bigcup_{\v\in |T|} \{\v\}\times U'_\v
\end{equation}
over $T$ such that $F(\v, z)=(\sigma(\v), f^{r(\v)}(z))$ for each $z\in U_\v$ and such that the filled Julia set of $F$ is the union of $\{\v\}\times K(\v,f)$:
$$K(\v,f)=\bigcap_{\theta\sim_{\lambda(f_0)} \theta'} \overline{S(\theta, \theta'; \v)}\cap K(f),$$
where $S(\theta,\theta';\v)$ is the component of $\C\setminus (\overline{\mathcal{R}_f(\theta)\cup \mathcal{R}_f(\theta')})$ which contains external rays $\mathcal{R}_f(t)$ for which $\mathcal{R}_{f_0}(t)$ land on $\partial \v$.
We shall call such an AGPL map $F$ as in (\ref{eqn:lambda0R}) a {\em $\lambda(f_0)$-renormalization of } $f$.
While there are many choices of $U_\v$ and $U'_\v$, the hybrid class of $\lambda(f_0)$-renormalizations of $f$ is uniquely determined.

In order to choose an external marking for $F$, Inou-Kiwi first fixed an {\em internal angle system} which is, by definition, a collection of homeomorphisms $\alpha=(\alpha_{\v}:\partial \v \to \mathbb R/\mathbb Z)_{\v\in |T|}$ such that:
\[\delta(\v)\alpha_{\v}=\alpha_{\sigma(\v)}\circ f_0^{r(\v)}\mod 1.\]
An internal angle system is uniquely determined by the points $z_\v=\alpha_{\v}^{-1}(0)$, $\v\in |T|$, which are (pre-)periodic points of $f_0$. Choose for each $\v\in |T|$ an external angle $\theta_\v$ so that $\mathcal{R}_{f_0}(\theta_\v)$ lands at $z_\v$. They observed that the external rays $\mathcal{R}_f(\theta_\v)$ define an external marking of $F$ and this external marking is independent of the choices of $\theta_\v$.
Indeed, we can choose $(\theta_\v)_{\v\in |T|}$ such that $\delta_{\v}\theta_\v=\theta_{\sigma(\v)}\mod 1$ for each $\v\in |T|$, see Lemma~\ref{lem:externalangle}.

The IK straightening map $\chi:\mathcal{C}(f_0)\to \mathcal{C}(T)$ is defined as follows. For each $f\in \mathcal{C}(f_0)$, $\chi(f)$ is the generalized polynomial in $\mathcal{C}(T)$ for which there is a hybrid conjugacy from a $\lambda(f_0)$-renormalization of $f$ to $\chi(f)$ sending the external marking determined by the internal angle system to the standard marking for $\chi(f)$.

\begin{Lemma}\label{lem:externalangle} Let $f_0$ be as above and let $\alpha$ be an internal angle system. Then there exists $\theta_\v\in \Q/\Z$, $\v\in |T|$, such that $\mathcal{R}_{f_0}(\theta_\v)$ lands at $\alpha_\v^{-1}(0)$ and such that $f_0^{r(\v)}(\mathcal{R}_{f_0}(\theta_\v))=\mathcal{R}_{f_0}(\theta_{\sigma(\v)})$.
\end{Lemma}
\begin{proof} {\bf Claim.} Let $W$ be a fixed Fatou component of $f\in \mathcal{C}_d$ and let $p$ be a repelling fixed point of $f$ in $\partial U$. If the external ray $\mathcal{R}_f(t)$ lands at $p$, then $dt=t\mod 1$.

Let $\Theta\subset \mathbb{R}/\mathbb{Z}$ denote the set of external angles $\theta$ for which $\mathcal{R}_f(\theta)$ lands at $p$. It is well known that $m_d: t\mapsto dt\mod 1$ maps $\Theta$ onto itself and preserves the cyclic order. So all the angles $\theta\in \Theta$ has the same period under $m_d$. We may certainly assume that $\Theta$ contains at least two points.  Let $\Gamma$ denote the union of these external rays and $\{p\}$. Let $V$ denote the component of $\C\setminus \Gamma$ which contains $W$ and let $V_1$ denote the component of $\C\setminus f^{-1}(\Gamma)$ which contains $W$. Then $\partial V\subset \partial V_1$ and $f:V_1\to V$ is a proper map.  It follows that $f$ must fix both of the external rays in $\partial V$.  This proves the claim.

Now for each $\v\in |T|$, we have
$f^{r(\v)}(\alpha_{\v}^{-1}(0))=\alpha_{\sigma(\v)}^{-1}(0)$. So each $\alpha_{\v}^{-1}(0)$ eventually lands at a repelling periodic orbit. The claim enables us to choose $\theta(\v)\in \mathbb{R}/\mathbb{Z}$ such that $\delta_\v\theta_\v=\theta_{\sigma(\v)}\mod 1$ for each $\v\in |T|$.
\end{proof}

\section{Yoccoz puzzle}\label{sec:puzzle}
Let $f_0$ be a monic centered postcritically finite hyperbolic and primitive polynomial of degree $d\ge 2$. In this section, we shall construct a specific Yoccoz puzzle for $f_0$ (Theorem~\ref{thm:puzzle}).
Recall that $\Poly_d$ denotes the collection of monic centered polynomials of degree $d$.

We say that a finite set $Z$ is {\em $f_0$-admissible} if the following hold:
\begin{itemize}
\item $f_0(Z)\subset Z$,
\item each periodic point in $Z$ is repelling;
\item for each $z\in Z$,  there exist at least two external rays landing at $z$.
\end{itemize}
Let $\Gamma_0=\Gamma^Z_0$ denote the union of all the external rays landing on $Z$, the set $Z$ itself and the equipotential $\{G_{f_0}(z)=1\}$.
For each $n\ge 1$, define $\Gamma^Z_n=f_0^{-n}(\Gamma^Z_0)$.
A bounded component of $\C\setminus \Gamma^Z_n$ is called a {\em $Z$-puzzle piece} of depth $n$.

The aim of this section is to prove the following:

\begin{Theorem}\label{thm:puzzle} Let $f_0$ be a monic centered polynomial of degree $d\ge 2$ which is postcritically finite, hyperbolic and primitive. Assume that $f_0(z)\not=z^d$. Then there exists an $f_0$-admissible  finite set $Z$ such that for each (finite) critical point $c$ of $f_0$, if $Y_n(c)$ denotes the $Z$-puzzle piece of depth $n$ which contains $c$ and $U(c)$ denotes the Fatou component containing $c$, then $\bigcap_{n=0}^\infty Y_n(c)=\overline{U(c)}$.
\end{Theorem}

We say a point $a\in J(f_0)$ is {\em bi-accessible} if it is the common landing points of two distinct external rays.
A point in $J(f_0)$ is called {\em buried} if it is not in the boundary of any bounded Fatou componnet.
We shall need the following lemmas to find buried bi-accessible periodic points.

\begin{Lemma}\label{lem:puuzle-1} Let $f_0\in \text{Poly}_d$ be a postcritically finite hyperbolic polynomial with $f_0(z)\not=z^d$, where $d\ge 2$.  Then any bi-accessible point in the boundary of a bounded Fatou component is eventually periodic.
\end{Lemma}
\begin{proof} Arguing by contradiction, assume that there exists a bi-accessible point $a_0$ which is in the boundary of a bounded Fatou component $U$ and such that $a_0$ is not eventually periodic. By Sullivan's no wandering theorem, $U$ is eventually periodic. So it suffices to consider the case that $U$ is fixed by $f_0$.

Let $t, t'\in \R/\Z$, $t\not=t'$, be such that $\mathcal{R}_{f_0}(t)$ and $\mathcal{R}_{f_0}(t')$ land at $a_0$. For each $n\ge 1$, let $a_n:=f_0^n(a_0)$ and let $t_n=d^n t$, $t'_n=d^n t'$. Then $a_n$ are pairwise distinct, $t_n, t'_n\not\in \Q/\Z$ and $t_n\not=t'_n$ for all $n\ge 0$. Let $\Gamma_n=\mathcal{R}_{f_0}(t_n)\cup \mathcal{R}_{f_0}(t'_n)\cup \{a_n\}$ and let $W_n$ and $W'_n$ be the components of $\C\setminus\Gamma_n$ so that $W'_n\supset U$
and $W_n\cap U=\emptyset$. Note that $\overline{W_n}\cap \overline{U}=\{a_n\}$ for each $n\ge 0$.
Since $\partial W_n$, $n\ge 0$, are pairwise disjoint, $W_n$, $n\ge 0$, are pairwise disjoint. So there exists $n_0$ such that for all $n\ge n_0$, $W_n$ contains no critical point.

{\em Claim.} If $W_n$ contains no critical point, then $f_0(W_n)=W_{n+1}$.

To see this, let $\widehat{\Gamma}_n=f_0^{-1}(\Gamma_{n+1})$ which is a finite union of simple curves stretching to infinity on both sides. Let $V_{n}$ (resp. $V'_n$) denote the component of $\C\setminus \widehat{\Gamma}_n$ which contains $a_n$ in its boundary and such that $V_n\subset W_n$ and $\partial W_n\subset \partial V_n $ (resp. $V'_n\subset W'_n$ and  $\partial W'_n\subset \partial V'_n $).
Then $f_0(V_n)$ and $f_0(V'_n)$  are distinct components of $\C\setminus \Gamma_{n+1}$. Since $V'_n\supset U$, we have $f_0(V'_n)=W'_{n+1}$ and hence $f_0(V_n)=W_{n+1}$. Since $W_n$ contains no critical point, $f_0:V_n\to W_{n+1}$ is a conformal map, which implies that $\partial V_{n}$ consists of one component of $\widehat{\Gamma}_n$. Thus $\partial V_n=\partial W_n$, hence $W_n=V_n$, $f_0(W_n)=W_{n+1}$.

Thus, for all $n\ge n_0$, $f_0(W_n)=W_{n+1}$. It follows that $f^n_0(W_{n_0})=W_{n+n_0}$ for all $n\ge 0$. So $W_{n_0}$ is a wandering domain, which contradicts Sullivan's no wandering theorem. A more elementary argument is as follows: We can choose $\theta\in \mathbb Q/\mathbb Z$ such that $\mathcal R_{f_0}(\theta) \subset W_{n_0}$. As $\mathcal R_{f_0}(\theta)$ is eventually periodic, $W_{n_0}$ cannot be a wandering domain.
\end{proof}
\begin{Lemma}\label{lem:puzzle0}
Let $f_0\in \text{Poly}_d$ be a postcritically finite hyperbolic polynomial with $f_0(z)\not=z^d$, where $d\ge 2$.
Then $f_0$ has a bi-accessible repelling periodic point.
\end{Lemma}
\begin{proof}
Without loss of generality, we assume that $f_0$ has a fixed bounded Fatou component $U$, and
let $$\Lambda=\{t\in \R/\Z: \mathcal{R}_{f_0}(t) \textrm{ lands on } \partial U\}.$$
Since the Julia set of $f_0$ is locally connected, by the Caratheodory Theorem, there is a continuous surjective map $\pi:\R/\Z\to J(f)$ such that
the external ray $\mathcal{R}_{f_0}(t)$ lands at $\pi(t)$ and hence $\pi\circ m_d=f_0\circ \pi$, where $m_d: \R/\Z\to \R/\Z$, $t\mapsto d t\mod 1$. In particular, $\Lambda=\pi^{-1}(\partial U)$ is a non-empty compact subset of $\R/\Z$ which is forward invariant under the map $m_d$. Since $f_0(z)\not=z^d$, $J(f)$ is not a Jordan curve, so $\Lambda$ is a proper subset of $\R/\Z$. Thus $\pi: \Lambda\to \partial U$ is not a homeomorphism. Since $\pi: \Lambda\to \partial U$ is continuous and surjective, this map is not injective. In other words, there exists a bi-accessible point $w\in \partial U$. By Lemma~\ref{lem:puuzle-1}, $w$ is eventually periodic. Any periodic point in the orbit of $w$ is a bi-accessible repelling periodic point.
\end{proof}
\begin{Lemma}\label{lem:puzzle} If $f_0\in Poly_d$ is postcritically finite, hyperbolic and primitive and $f_0(z)\not=z^d$, then $f_0$ has a buried biaccessible repelling periodic point.
\end{Lemma}
\begin{proof} In order to show that $f_0$ has a buried bi-accessible point, it is enough to show that for some $s\ge 1$, $f_0^s$ has a repelling fixed point which is the landing point of an external ray not fixed by $f_0^s$. Indeed, if a repelling fixed point $p$ of $f_0^s$ is in the boundary of a bounded Fatou component $V$, then by the assumption that $f_0$ is primitive, we have $f_0^s(U)=U$, hence by the Claim in Lemma~\ref{lem:externalangle}, any external ray landing at $p$ is fixed by $f_0^s$. Therefore, it is enough to prove the following Statement $N$ for each $N\ge 0$.

{\bf Statement $N$.} Suppose that $f_0\in \text{Poly}_d$ is a postcritically finite, hyperbolic and primitive map and $f_0(z)\not=z^d$, where $d\ge 2$. If $f_0$ has at most $N$ attracting periodic points, then there exists $s\ge 1$ such that $f_0^s$ has a repelling fixed point $p$ which is the landing point of an external ray which is not fixed by $f_0^s$.

We proceed by induction on $N$. Statement $0$ is a null statement.
Let $N\ge 1$ and assume that the statement $N'$ holds all $0\le N'<N$.  Let $f_0\in Poly_{d}$ be as in Statement $N$.
By Lemma~\ref{lem:puzzle0}, $f_0$ has a bi-accessible repelling periodic point $p_0$. Let $\mathcal{R}_{f_0}(t_i)$, $1\le i\le q$, be the external rays landing at $p_0$, where $q\ge 2$. Replacing $f_0$ by an iterate of $f_0$, we assume that
\begin{itemize}
\item all the external rays $\mathcal{R}_{f_0}(t_i)$ are fixed by $f_0$;
\item all attracting periodic points of $f_0$ are fixed by $f_0$.
\end{itemize}
In particular, $f_0(p_0)=p_0$. (Note that $f_0^k$, $k\ge 1$, satisfies the assumption of Statement $N$.)
Let us construct a Yoccoz puzzle using $Z=\{p_0\}$. Let $Y_0^j$, $1\le j\le q$ denote the puzzle pieces of depth zero and for each $n\ge 1$, let $Y_n^j$ denote the puzzle piece of depth $n$ which satisfies $Y_n^j\subset Y_0^j$
and $p_0\in \overline{Y_n^j}$. Since all the external rays $\mathcal{R}_{f_0}(t_j)$ are fixed, $f_0: Y_n^j\to Y_{n-1}^j$ is a proper map. Let $d_{n,j}$ denote the degree of $f_0:Y_n^j\to Y_{n-1}^j$, let $D_{n,j}=d_{1,j}d_{2,j}\cdots d_{n,j}$ and let $d_j=\lim_{n\to\infty} d_{n,j}$. Let $\Delta_{n,j}=\{t\in\R/\Z: \mathcal{R}_{f_0}(t)\cap \overline{Y_n^j}\not=\emptyset\}$. Note that $\Delta_{0,j}$ is a closed interval and $\Delta_{n,j}$ is a disjoint union of $D_{n,j}$ closed intervals each of which is mapped onto $\Delta_{0,j}$ under $m_d^n$ diffeomorphically.

Let us show $d_{1,j}\ge 2$ for each $j$. Indeed otherwise, $f_0: Y_1^j\to Y_0^j$ is a conformal map, which implies that $m_d:\Delta_{1,j}\to \Delta_{0,j}$ is a homeomorphism, which is absurd since $\Delta_{1,j}$ intersects both endpoints of $\Delta_{0,j}$.

{\em Case 1.} Suppose that $d_j=1$ holds for some $j$. Let $s_0$ be such that $d_{s,j}=1$ for all $s> s_0$. Then $f_0^{s-s_0}|Y_s^j$ is univalent for all $s> s_0$. Since $f_0$ has only finitely many attracting periodic points and every attracting cycle of $f_0$ contains a critical point, there exists $s_1>s_0$ such that $\overline{Y_{s_1}^j}$ does not contain any attracting periodic point of $f_0$.
Consider the map $f_0^{s_1}: \overline{Y_{s_1}^j}\to \overline{Y_0^j}$ which has degree $D:=D_{s_1,j}\ge d_{1,j}\ge 2$. By the thickening technique (\cite{Mil4}), it extends to a polynomial-like map $f_0^{s_1}: U_j\to U_j'$ of degree $D$ in the sense of Douady and Hubbard, so it has $D$
fixed points which are contained in $\overline{Y_{s_1}^j}$. By our choice of $s_1$, none of these fixed point is attracting. Since $f_0$ is hyperbolic, it follows that all the $D$ fixed points of $f_0^{s_1}: U_j\to U_j'$ are repelling.
The number of external rays of $f_0$ which intersect $\overline{Y_{s_1}^j}$ and are fixed by $f_0^{s_1}$ is exactly $D$, with two of them landing at the same point $p_0$. It follows that one of the repelling fixed point $p$ of $f_0^{s_1}|\overline{Y_{s_1}^j}$ is not the landing point of $f_0^{s_1}$-fixed external ray.

{\em Case 2.} Assume that $d_j>1$ holds for all $j$. Take $n_j$ sufficently large so that $d_{n,j}=d_j$ for all $n\ge n_j$. Then all critical points of $f_0|Y_{n_j}^j$ do not escape from $Y_{n_j}^j$ under iteration, hence $f_0: \overline{Y_{n_j}^j}\to \overline{Y_{n_j-1}^j}$ is a proper map of degree $d_j$ with non-escaping critical points. Again by the thickening technique (\cite{Mil4}), it extends to a polynomial-like map $f_0: U_j\to U_j'$ of degree $d_j$ in the sense of Douady and Hubbard.  Thus it is topologically conjugate to a polynomial $g_j\in Poly_{d_j}$ with connected Julia set. The polynomial $g_j$ is hyperbolic, postcritically finite and primitive. Since for any $j'\not=j$, $f_0$ has an attracting fixed point in $Y_{n_{j'}}^{j'}$, the number of attracting periodic points of $f_0: U_j\to U'_j$ is less $N$. Thus the number of attracting periodic points of $g_j$ is less than $N$.

{\em Subcase 2.1} Assume $g_j(z)\not=z^{d_j}$ for some $j$. Then by the induction hypothesis, $g_j$ has a periodic point $\tilde{p}$ of period $s$ which is not the landing point of $g_j^s$-fixed external rays. Taking the corresponding periodic point $p$ of $f_0: U_j\to U_j'$, we are done.

{\em Subcase 2.2} Assume $g_j(z)=z^{d_j}$ for all $j$. This implies that the filled Julia set of $f_0: U_j\to U_j'$ is the closure of a Jordan disk $V_j$ which contains $p_0$. Each $V_j$ is a bounded Fatou component of $f_0$. These bounded Fatou components $V_j$, $1\le j\le q$, have $p_0$ as a common point in their closures, contradicting the assumption that $f_0$ is primitive.
\end{proof}

\newcommand{\Crit}{\text{Crit}}
\begin{proof}[Proof of Theorem~\ref{thm:puzzle}]
Let $\Crit_{per}$ denote the set of all periodic critical points of $f_0$ and for each $c\in \Crit_{per}$, let $q(c)$ denote its period. For each admissible set $Z$ and $c\in\Crit_{per}$, let $s_Z(c)$ denote the minimal positive integer such that $f_0^{s_Z(c)}(c)\in \bigcap_{n=0}^\infty Y_n^Z(c)$ and let $d_Z(c)=\lim_{n\to\infty} \text{deg} (f^{s_Z(c)}|Y_n^Z(c))$. Of course $q(c)\ge s_Z(c)$.
Note that if $Z\subset Z'$ then $s_Z(c)\le s_{Z'}(c)$ for all $c\in\Crit_{per}$, and if $s_Z(c)=s_{Z'}(c)$ then $d_{Z}(c)\ge d_{Z'}(c)$.
Given admissible sets $Z\subset Z'$ we say that $Z'$ is a {\em (proper) refinement} of $Z$ if one of the following holds:
\begin{itemize}
\item there exists $c_0\in \Crit_{per}$ such that $s_{Z'}(c_0)>s_{Z}(c_0)$;
\item $s_{Z'}(c)=s_Z(c)$ for all $c\in \Crit_{per}$ and there exists $c_0\in \Crit_{per}$ such that $d_{Z'}(c_0)< d_{Z}(c_0)$.
\end{itemize}
Clearly, there does not exist an infinite sequence of admissible sets $\{Z_n\}_{n=1}^\infty$ such that for all $n$, $Z_{n+1}$ is a refinement of $Z_n$.

Let us say an $f_0$-admissible set $Z$ is {\em buried} if $Z$ is disjoint from the boundary of any bounded Fatou component. A buried $f_0$-admissible set exists by Lemma~\ref{lem:puzzle}. It suffices to prove that if $Z$ is a buried $f_0$-admissible set for which the property required by the theorem does not hold, then there exists a buried $f_0$-admissible set $Z'$ which is a refinement of $Z$.

To this end, assume that there exists $c_0\in \Crit_{per}$ such that $\bigcap_{n=0}^\infty Y_n^Z(c_0)\supsetneq \overline{U(c_0)}$.
Write $s=s_Z(c_0)$. When $N$ is large enough, the critical points of the proper map $g=f_0^s|Y_{N+s}(c_0)$ never escapes from its domain. Using the thickening technique (\cite{Mil4}), $g$ extends to a Douady-Hubbard polynomial-like map with connected Julia set. Thus it is hybrid equivalent to a monic centered polynomial $P$ which is necessarily hyperbolic and postcritically finite. Let $D\ge 2$ denote the degree of $P$ and let $h$ denote a hybrid conjugacy. As the filled Julia set of $P$ is not a topological disk, $P(z)\not=z^D$. So by Lemma~\ref{lem:puzzle}, $P$ has a repelling periodic point $\hat{z}_1$ which is biaccessible and buried.
By~\cite[Lemma 3.6]{I1} (see also ~\cite[Theorem 7.11]{McM1}), $z_1=h^{-1}(\hat{z}_1)$ is a buried biaccessible repelling periodic point of $f_0$.
Let $Z'$ denote the union of $Z$ and the $f_0$-orbit of $z_1$. As $\bigcap_{n=0}^\infty Y_n^{Z'}(c_0)$ is a proper subset of $\bigcap_{n=0}^\infty Y_n^Z(c_0)$, either $s_{Z'}(c_0)\not=s_Z(c_0)$ or $d_{Z'}(c_0)< d_Z(c_0)$. This completes the proof.
\end{proof}

We shall need the following result later.
\begin{Proposition}\label{prop:noncrpiece} Let $f_0$ and $Z$ be as in Theorem~\ref{thm:puzzle}. Then
$$\sup\{\text{diam}(Y): Y \text{ is a puzzle piece of depth }n, \overline{Y}\cap Z\not=\emptyset\}\to 0\text{ as } n\to\infty.$$
\end{Proposition}
\begin{proof} For each $n\ge 0$ and $z\in Z$, let $Y^*_n(z)$ denote the union of the closures of the puzzle pieces of depth $n$ which contain $z$ in their boundaries. Since $Z$ is finite, there exists $N$ such that $Y^*_N(z)\cap P(f_0)=\emptyset$ for all $z\in Z$. For each $n\ge 0$, and $z\in Z$, $f_0^n: Y_{n+N}^*(z)\to Y_N^*(f_0^n(z))$ is a conformal map which extends to a definite neighborhood of $f_0^n(z)$. It follows that $f_0^n|Y_{n+N}^*(z)$ has uniformly bounded distortion. Since $z\in J(f)$, this implies that $\text{diam}(Y_n^*(z))\to 0$ as $n\to\infty$.
\end{proof}

We shall construct $\lambda(f_0)$-renormalizations using the puzzle given by Theorem~\ref{thm:puzzle}.
The following is a criterion which will be used in the proof of surjectivity part of the main theorem.

\begin{Proposition}\label{prop:com}
Let $N_0$ be a positive integer such that for each $\v\in |T|$,
the puzzle pieces $f_0^j(Y_{N_0}(\v))$, $\v\in |T|$, $1\le j\le r(\v)$, are pairwise disjoint, and
$f_0^{r(\v)}: Y_{N_0}(\v)\to Y_{N_0-r(\v)}(\sigma(\v))$ has degree $\delta(\v)$.
Assume that $f\in \text{Poly}_d$ satisfies the following:
\begin{enumerate}
\item there is a homeomorphism $\psi: \C\to \C$ with the following properties:
\begin{itemize}
\item $\phi_f\circ \psi=\phi_{f_0}$ holds near $\infty$, where $\phi_f$ and $\phi_{f_0}$ are the B\"ottcher map for $f$ and $f_0$ respectively;
\item $\psi \circ f_0(z)=f\circ \psi(z)$ for all $z\in \C\setminus \bigcup_{\v} Y_{N_0}(\v)$.
\end{itemize}
\item  The map
$$F:\bigcup_{\v}\{\v\}\times \psi(Y_{N_0}(\v))\to \bigcup_{\v} \{\v\}\times \psi(Y_{N_0-r(\v)}(\sigma(\v))),$$
defined as $F(\v, z)=(\sigma(\v),f^{r(\v)}(z))$, is an AGPL map with fibrewise connected filled Julia set.
\end{enumerate}
Then $f\in \mathcal{C}(f_0)$ and $F$ is a $\lambda(f_0)$-renormalization of $f$.
\end{Proposition}
\begin{proof}
First of all, note the assumption implies that the filled Julia set $K(f)$ of $f$ is connected and $\widehat{Z}=\psi(Z)$ is an admissible set for $f$. It suffices to show that $f\in \mathcal{C}(f_0)$. Once this is proved, the other statement follows from \cite[Proposition 3.13]{IK}.
Let $L(\v, f)$ denote the filled Julia set of $F$ in the fiber $\{\v\}\times \C$.

{\bf Step 1.} We show by induction that for each $k\ge N_0$, there is a homeomorphsim $\psi_k:\C\to\C$ which coincide with $\phi_f^{-1}\circ \phi_{f_0}$ on $\Gamma_k\setminus J(f)$.

For $k=N_0$, we choose $\psi_{N_0}=\psi$.
Assume now that $\psi_k$ has been defined for some $k\ge N_0$ and let us construct $\psi_{k+1}$. For each $Y\subset \C$, denote $Y'=\psi_k(Y)$. It suffices to construct, for each $Y\in \mathcal{Y}_k$, a homeomorphism
$\psi_{k+1}: \overline{Y}\to \overline{Y'}$ so that $f\circ \psi_{k+1}(z)=\psi_k\circ f_0(z)$ for $z\in \overline{Y}\cap \Gamma_{k+1}$. Indeed, if $Y$ does not contain a critical point of $f_0$, then $f_0:Y\to f_0(Y)$ is a conformal map, and so is $f: Y'\to f(Y')$. In this case, we define $\psi_{k+1}|\overline{Y}=(f|\overline{Y'})^{-1}\circ (\psi_k|f_0(\overline{Y})) \circ (f_0|\overline{Y}))$. Assume that $Y$ contains a critical point of $f_0$, so that $Y=Y_k(\v)$ for some $\v\in |T|$ and hence $Y'\supset L(\v, f)$. Let $B=Y_{k-r(\v)+1}(\sigma(\v))$, $A=Y_{k+1}(\v)$ and $X=Y_{k-r(\v)}(\sigma(\v))$. Then $B'\supset L(\sigma (\v), f)$, $A'\supset L(\v, f)$ and $X'\supset L(\sigma(\v), f)$. Since $f_0^{r(\v)}: \overline{Y}\setminus A\to \overline{X}\setminus B$ and $f^{r(\v)}: \overline{Y'}\setminus A'\to \overline{X'}\setminus B'$ are both $\delta(\v)$ to $1$ covering, there is a homeomorphism $\psi_{k+1}: \overline{Y}\setminus A\to \overline{Y'}\setminus A'$ such that $\psi_{k+1}\circ f_0^{r(\v)}=f^{r(\v)}\circ \psi_k$ on $\overline{Y}\setminus A$ and $\psi_{k+1}=\psi_k$ on $\partial Y$. Extending the map $\psi_{k+1}$ in an arbitrary way to a homeomorphism from $\overline{Y}$ to $\overline{Y'}$, we obtain the desired map $\psi_{k+1}:\overline{Y}\to \overline{Y'}$.

{\bf Step 2.} For each $k\ge N_0$, there is a qc map $\Psi_k$ such that $\Psi_k=\phi_f^{-1}\circ \phi_{f_0}$ near infinity and such that $f\circ \Psi_k(z)=\Psi_k\circ f_0(z)$ for all $z\not\in \bigcup_\v Y_k(\v)$. This is well-known. See for example~\cite[Section 5]{KSS}.
This implies that if $\mathcal{R}_{f_0}(\theta_1)$ and $\mathcal{R}_{f_0}(\theta_2)$ land at a common point which is not in
$\bigcup_{n=0}^\infty f_0^{-n}\left(\bigcup_{\v\in |T|}\partial \v\right)$ ($\theta_1,\theta_2\in \mathbb{Q}/\mathbb{Z}$), then $\mathcal{R}_f(\theta_1)$ and $\mathcal{R}_{f}(\theta_2)$ have a common landing point as well. Indeed, there exists $k$ such that the whole $f_0$-orbit of the rays $\mathcal{R}_{f_0}(\theta_1)$ and $\mathcal{R}_{f_0}(\theta_2)$ lie outside $\bigcup_\v Y_k(\v)$, so
$\mathcal{R}_f(\theta_i)=\Psi_k(\mathcal{R}_{f_0}(\theta_i))$, $i=1,2$.

{\bf Step 3.} It remains to show that if $\mathcal{R}_{f_0}(\theta_1)$ and $\mathcal{R}_{f_0}(\theta_2)$, $\theta_1, \theta_2\in\mathbb{Q}/\mathbb{Z}$ landing at a common point in $\bigcup_{n=0}^\infty f_0^{-n}\left(\bigcup_{\v\in |T|}\partial \v\right)$,
then $\mathcal{R}_f(\theta_1)$ and $\mathcal{R}_{f}(\theta_2)$ have a common landing point.

Let us first assume that the common landing point is in $\partial \v_0$ for some $\v_0\in |T|$. Let $\Psi=\Psi_{N_0}$ be given by Step 2. We define a new qc map $H$ from $\mathbb{C}\setminus \bigcup_{\v}\overline{\v} \to \mathbb{C}\setminus \bigcup_{\v} L(\v, f)$ such that
$H=\Psi$ outside $\bigcup_{\v} Y_{N_0}(\v)$ and such that $H\circ f_0^{r(\v)}=f^{r(\v)}\circ H$ inside $\bigcup_{\v} Y_{N_0}(\v)$. Note that $H$ maps $\mathcal{R}_{f_0}(\theta_i)$ onto $\mathcal{R}_f(\theta_i)$, $i=1,2.$ For each $\textbf{v}\not=\textbf{v}_0$, choose a quasidisk $\Omega_{\textbf{v}}$ so that these quasidisks are pairwise disjoint and disjoint from $\textbf{v}_0\cup \mathcal{R}_{f_0}(\theta_1)\cup\mathcal{R}_{f_0}(\theta_2)$.  Let $H_0=H$ on $\mathbb{C} \setminus ( \overline{\v_0}\cup\bigcup_{\v\ne \v_0}\Omega_{\v})$, and then extend $H_0$ quasiconformally to $\mathbb C\setminus \overline{\v_0}$ by Beurling-Ahlfors extension.
So we obtain a qc map $H_0: \mathbb{C}\setminus \overline{\v_0}\to \mathbb{C}\setminus L(\v_0,f)$ which again maps $\mathcal{R}_{f_0}(\theta_i)$ onto $\mathcal{R}_f(\theta_i)$, $i=1,2$. Let $\vartheta:\mathbb{C}\setminus \overline{\mathbb{D}}\to \mathbb{C}\setminus L(\v_0,f)$ denote a Riemann mapping. Since $\partial \v_0$ is a Jordan curve, $\vartheta^{-1}\circ H_0$ extends continuously to $\mathbb{C}\setminus \v_0$. Since $\mathcal{R}_f(\theta_i)$ both land  and $\vartheta^{-1}(\mathcal{R}_{f} (\theta_1))$ and $\vartheta^{-1}(\mathcal{R}_f(\theta_2))$ have
a common landing point, by Lindelof's theorem, we conclude that $\mathcal{R}_f(\theta_i)$, $i=1,2$, land at the same point.

For the general case, let $n\ge 1$ be minimal such that the common landing point of $\mathcal{R}_{f_0}(d^n \theta_i)$ ($i=1,2$) lie in $\bigcup_{\v} \partial \v$. As proved above, the external rays $\mathcal{R}_f(d^n \theta_i)$, $i=1,2$, have a common landing point $z$. Let $k$ be large integer such that the external rays $\mathcal{R}_{f_0}(d^j \theta_i)$, $0\le j<n$, lie outside $\bigcup_\v Y_k(\v)$,
let $Y$ denote $f_0$-the puzzle piece of depth $k+n$ which contains the common landing point of $\mathcal{R}_{f_0}(\theta_1)$ and $\mathcal{R}_{f_0}(\theta_2)$ and let $Y'=\psi_{k+n}(Y)$. Then $f^n: Y'\to Y_k(\v)$ is a conformal map for some $\v\in |T|$ and the rays $\mathcal{R}_f(\theta_1)$ and $\mathcal{R}_f(\theta_2)$ enter $Y'$. Thus both of them have to land at the unique point in $\overline{Y'}$ which is mapped to $z$ by $f^n$.
\end{proof}

\section{Kahn's quasiconformal distortion bounds}\label{sec:Kahn}
In this section, we will modify the argument in~\cite{Kahn}\footnote{According to Kahn, Yoccoz may have a similar result.} to obtain a K-qc extension principle. The main result is Theorem~\ref{thm:Kqcpuzzle} which will be used later to show the convergence of the Thurston Algorithm in the proof of the Main Theorem.

Throughout we fix a monic centered, hyperbolic, postcritically finite and primitive polynomial $f_0$ of degree $d$ such that $f_0(z)\not=z^d$. Let $Z$ be an admissible set given by Theorem~\ref{thm:puzzle} and let $Y_n(z)=Y_n^Z(z)$.
Let $\Crit(f_0)=\{c\in \C: f_0'(c)=0\}$ and let $L_n$ denote the domain of the first landing map to $\bigcup_{c\in \Crit(f_0)} Y_n(c)$:
\begin{equation}\label{eqn:dfnLn}
L_n=\left\{z\in \C: \exists k\ge 0 \text{ such that } f_0^k(z)\in \bigcup_{c\in \Crit(f_0)} Y_n(c)\right\}.
\end{equation}

\begin{Theorem}\label{thm:Kqcpuzzle}
There exists $N>0$ and for any puzzle piece $Y$ of depth $m\ge 0$,
there is a constant $C=C(Y)>1$ satisfying the following property: if $Q:Y \to Q(Y)$ is a qc map which is conformal a.e. in $Y\setminus L_{m+N}$,
then there exists a $C$-qc map $\widetilde Q :Y\to Q(Y)$ such that $\widetilde Q=Q$ on $\partial Y$.
\end{Theorem}

\subsection{Quasiconformal Distortion Bounds and a toy model}
The difficulty in proving the theorem is that the landing domains $L_{m+N}$ may come arbitrarily close to the boundary of $Y$. To deal with the situation, we shall need a toy model developed by Kahn (\cite{Kahn}).

Let us first recall some terminology from \cite{Kahn}. Let $U \subset \mathbb C$ be a Jordan domain and $A$ be a measurable subset of $U$. We say that $(A,U)$ has {\em bounded qc distortion} if there exists a constant $K\ge 1$ with the following property: if $Q:U\to Q(U)$ is a quasiconformal map and $\bar\partial Q=0$ a.e. outside $A$, then there is a $K$-q.c map $\tilde Q:U \to Q(U)$ such that $\tilde Q=Q$ on $\partial U$. Let  $\mathcal{QD}(A,U)$ denote the smallest $K$ satisfying the property.
Using this terminology, we can restate Theorem~\ref{thm:Kqcpuzzle} as follows:

\begin{theorem3'} There exists $N>0$ such that if
$Y$ is a puzzle piece $Y$ of depth $m\ge 0$,
\[\mathcal{QD}(L_{m+N}\cap Y, Y)<\infty.\]
\end{theorem3'}

We shall need the following easy facts.
\begin{Lemma}\label{compact}\cite[Fact~1.3.6]{Kahn} If $A \subset U$ is compact, then $\mathcal{QD}(A,U)<\infty$.
\end{Lemma}

\begin{Lemma}\label{qc1}\cite[Fact~1.3.4]{Kahn} Let $U$ and $V$ be Jordan domains in $\mathbb C$ and $A$ be a measurable subset of $U$.
If there exists a $L$-qc map $g:U\to V$ and $\mathcal{QD}(A,U)<\infty$, then \[\mathcal{QD}(g(A),V)\le L^2\mathcal{QD}(A,U).\]
\end{Lemma}

\begin{Lemma}\label{qc2} The following statements are equivalent:
\begin{enumerate}
\item [(i)] $\mathcal{QD}(A,U)=C<\infty$;
\item [(ii)] For any qc map $Q:U \to Q(U)$, if $\mathrm{Dil}(Q)\le K$ for some $K\ge 1$ a.e. outside $A$, then there is a $KC$-qc map $\tilde Q:U \to Q(U)$ such that $\tilde Q=Q$ on $\partial U$.
\end{enumerate}
\end{Lemma}

\begin{proof} It is obvious that (ii) implies (i). Let us show that (i) implies (ii).
Let $\mu$ be the Beltrami differential such that $\mu =\bar\partial Q^{-1}/\partial Q^{-1}$ on $Q(U\backslash A)$ and $\mu=0$ otherwise. By the Measurable Riemann Mapping Theorem, there exists $g:\mathbb C \to \mathbb C$ quasiconformal map with Beltrami differential $\mu$. Then $g\circ Q: U\to g\circ Q(U)$ is a quasiconformal map and $\bar\partial g\circ Q=0$ a.e. outside A. Thus there exists a $C$-qc map $G:U\to g\circ Q(U)$ such that $G=g\circ h$ on $\partial U$, where $C=\mathcal{QD}(A,U)<\infty$.
Finally, let $\tilde Q=g^{-1}\circ G$ and we are done.
\end{proof}

We shall now recall the {\em recursively notched square model} developed in~\cite{Kahn}.
Let $S=(0,1)\times (-1/2,1/2)$. Let $\mathcal{I}$ denote the collection of the components of $(0,1)\setminus \mathcal{C}$, where $\mathcal{C}$ is the ternary Cantor set. Let
\begin{equation}\label{eqn:setN}
\mathcal{N}=\bigcup_{I\in\mathcal{I}} \overline{I}\times [-|I|/2, |I|/2]
\end{equation}
which is a countable disjoint union of closed squares.
The following is ~\cite[Lemma 2.1.1]{Kahn}:
\begin{Theorem}\label{thm:Kahn}
$\mathcal{QD}(\mathcal{N}, S)<\infty.$
\end{Theorem}

\subsection{Reduce to the toy model}
We will work on the polynomial map $f_0$ fixed at the beginning of this section. In the following we write $\mathcal{R}(\theta)$ for $\mathcal{R}_{f_0}(\theta)$.
A {\em geometric ray-pair} is, by definition, a simple curve consisting of two distinct external rays together with their common landing point.
A {\em slice} is an open set $U$ bounded by two disjoint ray-pairs $\mathcal{R}(\theta_i)\cup \mathcal{R}(\theta_i')\cup \{a_i\}$, $i=1,2$ such that no external ray lying inside $U$ lands at either $a_1$ or $a_2$.

For every $z_0 \in Z$, the external rays landing at $z_0$ cut the complex plane $\mathbb C$ into finitely many sectors $S_1(z_0),\cdots,S_{n(z_0)}(z_0)$. Let $\mathcal S=\{S_j(z)\mid z\in Z,~1\le j\le n(z) \}$.  We list the elements in $\mathcal S$ as $S_1, S_2,\cdots, S_{\nu}$ where $\nu=\# \mathcal S$.  For each $j$, the boundary of $S_j$ is a geometric ray-pair:
there exists $\alpha^j\in Z$ and $\theta^-_j,\theta^+_j\in\mathbb{R}/\mathbb{Z}$ such that
$$\partial S_j=\mathcal R(\theta^-_j)\cup \{\alpha^j\}\cup \mathcal R(\theta^+_j).$$
We order $\theta^-_j, \theta^+_j$ in such a way that
$$\{t\in \mathbb{R}/\mathbb{Z}: \mathcal{R}(t)\subset S_j\}=(\theta_j^-, \theta_j^+).$$

\begin{Proposition}\label{prop:slice} For each $j\in \{1,2,\ldots, \nu\}$ and each $n$ sufficiently large, there exists a geometric ray-pair $\gamma_n^j=\mathcal{R}(t_n^-(j))\cup \mathcal{R}(t_n^+(j))\cup \{\alpha_n^j\}$ contained in $S_j$ with the following properties:
\begin{itemize}
\item $\alpha_n^j\in f_0^{-n}(Z)$;
\item $\theta^-_j, t_n^-(j), t_n^+(j), \theta^+_j$ lie in $\mathbb{R}/\mathbb{Z}$ in the anticlockwise order;
\item $\partial S_j$ and $\gamma_n^j$ bound a slice;
\item $t_n^-(j)\to \theta_j^-$, $t_n^+(j)\to \theta_j^+$ as $n\to\infty$.
\end{itemize}
\end{Proposition}

We postpone the proof of this proposition to the end of this section and show now how it implies Theorem~\ref{thm:Kqcpuzzle}.

\begin{proof}[Proof of Theorem 3']
For each $n$ large, let $\widehat{S}_n^j$ be the slice bounded by $\gamma_n^j$ and $\partial S_j$ given by Proposition~\ref{prop:slice}, and let $S_n^j=\{z\in \widehat{S}_n^j: G(z)<1/d^n\}$, where $G$ is the Green function of $f_0$. So $\overline{S_n^j}$ is a finite union of closures of puzzle pieces of depth $n$. Choose $N_*$ sufficiently large so that the closure of $S^j:=S^j_{N_*}$ is disjoint from the post-critical set of $f_0$. Let $q\gg N_*$ be a positive integer so that for any $j,j'$, the diameter of each component of $f_0^{-q}(S^j)$ is much smaller than that of $S^{j'}$.
For each $j\in \{1,2,\ldots,\nu\}$, $f_0^q$ maps a neighborhood $Q_j$ of $\alpha^j$ conformally onto the component of the interior of $\bigcup_{j'=1}^ {\nu} \overline{S^{j'}}$ which contains $f_0^q(\alpha_j)$. Let $A_j=Q_j\cap S^j$. Then $A_j$ is a quasi-disk which contains $\alpha^j$ in its boundary and there is $\sigma(j)\in \{1,2,\ldots, \nu\}$ such that $f_0^q(A_j)=S^{\sigma(j)}$. Similarly, there is a quasi-disk $B_j$ which is contained in $S^j$ and contains $\alpha_{N_*}^j$ in its boundary and $\tau(j)\in \{1,2,\ldots, \nu\}$ such that $f_0^q(B_j)=S^{\tau(j)}$.
Choosing $q$ large enough, we can ensure that $\overline{A_j}\cap\overline{B_j}=\emptyset$.
Note that $m_d^q(\theta^+_j)=\theta^+_{\sigma(j)}, m_d^q(\theta^-_j)=\theta^-_{\sigma(j)}$,
$m_d^q(t_{N_*}^-(j))=\theta^+_{\tau(j)}$ and $m_d^q(t_{N_*}^+(j))=\theta^-_{\tau(j)}$, where $m_d:\R/\Z\to \R/\Z$ denotes the map $t\mapsto dt \mod 1$.
Let $$F:\bigcup_{j=1}^{\nu} \overline{A_j}\cup\overline{B_j}\to \bigcup_{j=1}^{\nu} \overline{S^j}$$ be the restriction of $f_0^q$.
Let $A=(0,1/3)\times (-1/6, 1/6)$, $B=(2/3, 1)\times (-1/6,1/6)$ and $S=(0,1)\times (-1/2,1/2)$ and define a map
$$G: \bigcup_{j=1}^\nu \{j\}\times (A\cup B) \to \bigcup_{j=1}^\nu \{j\}\times S,$$
as follows:
$$G(j,z)=\left\{\begin{array}{ll}
(\sigma(j), 3z), &\mbox{ if } z\in A;\\
(\tau(j), 3(1-z)), &\mbox{ if } z\in B.
\end{array}
\right.
$$
Let $C_j=\overline{\{z\in S^j\setminus (\overline{A_j\cup B_j}): G(z)\le d^{-q-N_*}\}}$
and $C=[1/3,2/3]\times [-1/6, 1/6]$.

{\bf Claim.} There is a qc homeomorphism
$H_:\bigcup_{j=1}^\nu S^j\to \bigcup_{j=1}^\nu \{j\}\times S$ such that
\begin{enumerate}
\item[(i)] $H(A_j)=\{j\}\times A, H(B_j)=\{j\}\times B, H(C_j)=\{j\}\times C,$
\item[(ii)] $H\circ F=G\circ H$ holds on $\bigcup_{j=1}^\nu \overline {A_j\cup B_j}$.
\end{enumerate}
Indeed, it suffices to prove there is qc map
$H_0:\bigcup_{j=1}^\nu S^j\to  \bigcup_{j=1}^\nu \{j\}\times S$ such that (i) holds and (ii) holds on $\bigcup_{j} (\partial A_j\cup \partial B_j)$ (with $H$ replaced by $H_0$). Indeed, once such a $H_0$ is constructed, we can construct inductively a sequence $\{H_n\}_{n=0}^\infty$ of qc maps by pull-back which has the following properties:
\begin{itemize}
\item $H_{n+1}=H_n$ on $\bigcup_{j=1}^\nu S^j \setminus (A_j\cup B_j);$
\item $H_{n}\circ F=G\circ H_{n+1}$ holds on $\bigcup_{j=1}^\nu \overline {A_j\cup B_j}$.
\end{itemize}
These maps $H_n$ have the same maximal dilatation as $H_0$, and they eventually stablize for any point in the set $X=\{z\in \bigcup_{j=1}^\nu S^j: F^n(z)\not\in \bigcup A_j\cup B_j\mbox{ for some }n\}$. Since $F$ is uniformly expanding, the set $X$ is dense in $\bigcup_j S^j$, it follows that $H_n$ converges to a qc map $H$ which satisfies the requirements.
For the existence of $H_0$, a concrete construction of a homeomorphism  with the desired properties can be easily done using B\"ottcher coordinate with an extra property that it is qc in $S^j\setminus \overline{A_j\cup B_j\cup C_j}$. It can be made global qc since $S^j, A_j, B_j, C_j$ are all quasi-disks.

Now, let
$$\mathcal{N}_j=\{z\in S^j: \exists n\ge 1\text{ such that }F^n(z) \text{ is well-defined and belongs to} \bigcup_{j'} C_{j'}\}.$$
Note that for each $j$,
$H(\mathcal{N}_j)=\{j\}\times\mathcal{N},$
where $\mathcal{N}$ is as in (\ref{eqn:setN}). Therefore,
$$Q:=\max_{j=1}^\nu \mathcal{QD}(\mathcal{N}_j, S^j)<\infty.$$

Let $N=q+N_*$. Then any landing domain of $Y_{N}:=\bigcup_{c\in\Crit(f_0)} Y_N(c)$ does not intersect $\bigcup_{k=0}^{q-1}\bigcup_{j=1}^\nu f_0^k(\partial A_j\cup\partial B_j)$.
Therefore, $L_N\cap S^j\subset \mathcal{N}_j,$
so that $$\mathcal{QD}(L_N\cap S^j, S^j)\le \mathcal{QD}(\mathcal{N}_j, S^j)\le Q.$$

Now let $Y$ be an arbitrary Yoccoz puzzle piece of depth $m\ge 0$.
Similarly as in the construction for $A_j$ and $B_j$ above, for each $x\in \partial Y\cap J(f)$, there is a quasi-disk $V_x$ which is contained in $Y$ and contains $x$ in its closure such that $f_0^{m}$ maps $V_x$ onto $S^{j(x)}$ for some $j(x)\in \{1,2,\ldots, \nu\}$. Since $\overline{S^{j(x)}}$ is a finite union of the closure of puzzle pieces of depth $N_*$ and disjoint from the postcritical set of $f_0$, for each $k=0,1,\ldots, m-1$, $f^k(V_x)$ is a finite union of the closure of puzzle pieces of depth $N_*+m-k$ ($< m+N$) and does not contain a critical point of $f_0$, so $f^k(V_x)\cap Y_{m+N}=\emptyset$. It follows that $f_0^{m}$
maps $V_x\cap L_{m+N}$ onto $S^{j(x)}\cap L_{m+N}$. Therefore
$$\mathcal{QD}(L_{m+N}\cap V_x, V_x)= \mathcal{QD}(L_{m+N}\cap S^{j(x)}, S^{j(x)})\le \mathcal{QD}(L_N\cap S^{j(x)}, S^{j(x)})\le Q.$$
These $V_x$'s are pairwise disjoint since each of them is mapped onto of a component of $\bigcup_j S^j$ univalently under $f_0^{m}$
Noting that
$(Y\cap L_{m+N})\setminus \bigcup_{x\in \partial Y\cap J(f)}V_x$ is compactly contained in $Y$, we conclude that $\mathcal{QD}(L_{m+N}\cap Y, Y)<\infty$.
\end{proof}

\begin{proof}[Proof of Proposition~\ref{prop:slice}]
We denote by $Y^j_n$ the unique puzzle piece of depth $n$ which attaches $\alpha^j$ and is contained in the sector $S_j$.
By Proposition~\ref{prop:noncrpiece}, when $n$ is sufficiently large, $\overline{Y_n^j}$ is disjoint from the postcritical set of $f_0$. Fix $n_0$ large so that for each $j$, there exists $\alpha^j_{n_0} \in \overline Y^j_{n_0}\cap J(f_0)$ with $\alpha^j_{n_0} \not\in Z$.
Let $\mathcal R(s^-_j)$ and $\mathcal R(s^+_j)$ be the external rays landing at $\alpha^j_{n_0}$ which are the boundary curves of $Y^j_{n_0}$. We assume that $\theta^-_j, s^-_j, s^+_j, \theta^+_j$ lie in $\mathbb{R}/\mathbb{Z}$ in the anticlockwise cyclic order.

  For $n\ge n_0$, we define two angles $t^-_n(j)$ and $t^+_n(j)$ as following:
\[t^-_n(j):=\sup\{\theta\in(\theta^-_j,s^-_j)\mid \mathcal R(\theta)\cap Y^j_{n}\ne \varnothing\}\]
and
\[t^+_n(j):=\inf\{\theta\in(s^+_j,\theta^+_{j})\mid \mathcal R(\theta)\cap Y^j_{n}\ne \varnothing\}.\]

\medskip
{\bf Claim 1.} {\em For every $1\le j\le \nu$ and any $n\ge n_0$, the two external rays $\mathcal R(t^-_n(j))$ and $\mathcal R(t^+_n(j))$ land at the same point, denoted by $\alpha_n^j$.}

Let $\mathcal R^n(t)=\mathcal R(t)\cap \{z\mid G_{f_0}(z)<d^{-n}\}$.
Clearly, $\mathcal R^n(t^-_n(j))$ and $\mathcal R^n(t^+_n(j))$ are on the boundary of $Y^j_n$ and are closest rays (on the boundary of $Y^j_n$) to $\mathcal R(s^-_j)$ and $\mathcal R(s^+_j)$ respectively. First, we show that $\mathcal R(t^-_{n_0+1}(j))$ and $\mathcal R(t^+_{n_0+1}(j))$ land at the same point.

{\em Case 1.} $\alpha_{n_0}^j\in \overline{Y_{n_0+1}^j}$. Then $t^{\pm}_{n_0+1}(j)=s^{\pm}_j$, and so claim holds.

{\em Case 2.} $\alpha_{n_0}^j\not\in \overline{Y_{n_0+1}^j}$. Then there exists $t^-$ and $t^+$ such that
\begin{itemize}
\item $\mathcal{R}(t^{\pm})$ intersects the boundary of $Y_{n_0+1}^j$ with a common landing point $\alpha\not\in \{\alpha_{n_0}^j,\alpha^j\}$;
\item $\mathcal{R}(t^+)\cup \mathcal{R}(t^-)\cup\{\alpha\}$ separates $\alpha_{n_0+1}^j$ from $\alpha^j$.
\end{itemize}
Without loss of generality, assume $t^-\in (\theta_j^-, s_j^-)$. Then $t^+\in (s_j^+,\theta_j^+)$. So for any $t\in (t_-, s_j^-)$, $\mathcal{R}(t)$ is disjoint from $Y_{n_0+1}^j$, hence $t^-=t_{n_0+1}^j$. Similarly, $t^+=t_{n_0+1}^j$. Thus $t_{n_0+1}^-\sim_{\lambda_{f_0}} t_{n_0+1}^+$.

The general case can be proved similarly by induction.

It is clear that $\partial S_j$ and $\gamma_n^j=\mathcal{R}(t_-^n(j))\cup\mathcal{R}(t_+^n(j))\cup \{\alpha_n^j\}$ bound a slice for each $n\ge n_0$. So it remains to show

\medskip
{\bf Claim 2.} {\em The sequence $\{t^-_n(j)\}_{n>n_0}$ decreases monotonically to $\theta^-_j$ and $\{t^+_n(j)\}_{n>n_0}$ increases monotonically to $\theta^+_{j}$ for all $j$.}
\medskip

Since $\mathrm{diam}(Y^j_n)$ converges to $0$, $\alpha^j_n \to \alpha^j$ . Obviously, $\{t^-_n(j)\}$ is monotonically decreasing, thus it converges to some $\theta\in[\theta^-_j,s^-_j]$. If $\theta \ne \theta^-_j$, then $\alpha^j_n$ converges to the landing point of $\mathcal R(\theta)$ which is not equal to $\alpha^j$. This leads to a contradiction.
\end{proof}

\section{Thurston's Algorithm}\label{sec:Thurston}
\subsection{Thurston's Algorithm}
The Thurston algorithm was introduced by Thurston to construct rational maps that is combinatorially equivalent to a given branched covering of the 2-sphere. See \cite{DH3}.
The algorithm goes as follows. Let $\widetilde{f}:\CC\to \CC$ be a quasi-regular map of degree $d>1$.
Given a  qc map  $h_0:\CC\to \CC$, let $\sigma_0$ be the standard complex structure on $\CC$ and $\sigma=(h_0\circ\tilde f)^*\sigma_0$. By the Measurable Riemann Mapping Theorem, there exists a qc map $h_1:\CC\to \CC$ with $h_1^*\sigma_0=\sigma$ so that $Q_0:=h_0\circ \widetilde{f}\circ h_1^{-1}$ is a rational map of degree $d$. The qc map $h_1$ is unique up to composition with a M\"obius transofrmation. Applying the same argument to $h_1$ instead of $h_0$, we obtain a qc map $h_2$ and a rational map $Q_1$ of degree $d$ such that $Q_1\circ h_2=h_1\circ \widetilde{f}$. Repeating the argument, we obtain a sequence of normalized qc maps $\{h_n\}_{n=0}^\infty$ and a sequence of rational maps $\{Q_n\}_{n=0}^\infty$ of degree $d$ such that $Q_n\circ h_{n+1}=h_{n}\circ \widetilde{f}$. The question is to study the convergence of $\{h_n\}_{n=1}^\infty$ and $\{Q_n\}_{n=1}^\infty$ after suitable normalization.

In \cite{R-L}, Rivera-Letelier applied the algorithm to a certain class of quasi-regular maps which may have non-recurrent branched points with infinite orbits.
In this section, we shall modify his argument and prove a convergence theorem for a quasi-regular map $\widetilde{f}:\CC\to \CC$ where the irregular part has a nice Markov structure.

For simplicity, we shall assume that $\widetilde{f}$ satisfies the following: $\widetilde{f}^{-1}(\infty)=\infty$ and $\widetilde{f}$ is holomorphic in a neighborhood of $\infty$. Below, we shall use the terminology {\em quasi-regular polynomial} for such a map.

An open set $\mathcal{B}$ is called {\em nice} if each component of $\mathcal{B}$ is a Jordan disk and $\widetilde{f}^k(\partial \mathcal{B})\cap \mathcal{B}=\emptyset$ for each $k\ge 1$. Let
$$D(\mathcal{B})=\{z\in \C: \exists n\ge 1\text{ such that } \widetilde{f}^n(z)\in\mathcal{B}\}$$
denote the domain of the first entry map to $\mathcal{B}$.
We say that a nice open set $\mathcal{B}$ is {\em free} if
$$P(\widetilde{f})\cap\overline{\mathcal{B}}=\emptyset,$$
where
$$P(\widetilde{f})=\overline{\bigcup_{c\in \Crit(\widetilde{f})}\bigcup_{n\ge 1} \{\widetilde{f}^n(c)\}},$$
and $\Crit(\widetilde{f})$ denotes that set of the ramification points of $\widetilde{f}$. We say that an open set $\mathcal{B}$ is
{\em $M$-nice} if it is nice and for each component $B$ of $\mathcal{B}$, the following three conditions hold:
\begin{equation}\label{eqn:shapeB}
\text{diam}(B)^2\le M \text{area} (B);
\end{equation}
\begin{equation}\label{eqn:Bretsmall}
\frac{\mathrm{area}(B\setminus D(\mathcal{B}))}{\mathrm{area}(B)}>M^{-1};
\end{equation}
\begin{equation}\label{eqn:Bretqc}
\mathcal{QD}(D(\mathcal B)\cap B, B)\le M,
\end{equation}
where $\mathcal{QD}$ is as defined in \S\ref{sec:Kahn}.

\begin{Theorem}\label{thm:Thurston}
Let $\widetilde{f}:\CC\to \CC$ be a quasi-regular polynomial of degree $d\ge 2$ and let $A\subset \CC$ be a Borel set such that $\bar\partial \tilde f=0$ a.e. outside $A$. Assume that there is a free open set $\mathcal{B}$ which is $M$-nice for some $M<\infty$  and a positive integer $T$ such that
$$(\ast) \text{ for every }z\text{ and }n\ge 1, \text{ if }\widetilde{f}^j(z)\not\in\mathcal{B}\text{ for each }0\le j <n,\text{ then  }\#\{0\le k<n:  \widetilde f^k(z) \in A\}\le T.$$
Then there is a continuous surjection $h:\CC\to \CC$ and a rational map $f:\CC\to \CC$ of degree $d$ such that $f\circ h=h\circ \tilde f$. Moreover, there is a qc map $\lambda_0$ such that $\lambda_0(z)=h(z)$ whenever $z\not\in \bigcup_{n=0}^\infty \widetilde{f}^{-n}(\mathcal{B})$ and such that
$\bar{\partial} \lambda_0=0$ holds a.e. on the set $\{z\in \CC: \widetilde{f}^n(z)\not\in A, \forall n\ge 0\}$.
\end{Theorem}
The rest of this section is devoted to a proof of the theorem. Without loss of generality, we may assume that $\mathcal{B}\subset A$. Note that
${D}(\mathcal{B})\cup \mathcal{B}$ is compactly contained in $\C$.
Starting with $h_0=id$, we construct a sequence of qc maps $h_n$ as above, normalized so that $h_n(z)=z+o(1)$ near infinity.
Note that there is a neighborhood $V$ of $\infty$ such that $\widetilde{f}^{-1}(V)\subset V$, so that all the maps $h_n$ are conformal in $V$.

For a map $\varphi:\C\to \C$ and $z\in \C$, let
$$H(\varphi;z)=\limsup_{r\searrow 0}\frac{\sup_{|w-z|=r} |\varphi(w)-\varphi(z)|}{\inf_{|w-z|=r}|\varphi(w)-\varphi(z)|}.$$
So if $\varphi$ is differentiable at $z$ with a positive Jacobian, then
$$H(\varphi;z)=\frac{|\partial \varphi (z)|+|\overline{\partial} \varphi(z)|}{|\partial \varphi (z)|-|\overline{\partial} \varphi(z)|}.$$

Fix $K>1$ such that $\widetilde{f}$ is $K$-quasi-regular. Let us call a point $z\in \C$ {\em regular} if the following hold:
\begin{enumerate}
\item $\widetilde{f}^n(z)$ is not a ramification point of $\widetilde{f}$ for any $n\ge 0$;
\item for any non-negative $n$ and $m$, the maps $h_n$ and $\widetilde{f}$ are differentiable at $\widetilde{f}^m(z)$ with a positive Jakobian. Moreover,
    $$H(\widetilde{f}; \widetilde{f}^m(z))\le K.$$
\end{enumerate}
Note that Lebesgue almost every point in $\C$ is regular.

\begin{Lemma}\label{lem:hkh-n1st}
If $z$ is a regular point, then for any integers $k>n\ge 0$, the following hold:
\begin{equation}
H(h_k\circ h_n^{-1}, h_n(z))=H(\widetilde{f}^{k-n};\widetilde{f}^n(z))\le \prod_{j=n}^{k-1} H(\widetilde{f}; \widetilde{f}^j(z))\le K^{\#\{n\le j<k: \widetilde{f}^j(z)\in A\}}.
\end{equation}
\end{Lemma}
\begin{proof}
The equality follows from the identity
$$\widetilde{f}^{k-n}\circ (Q_0\circ Q_1\circ \cdots Q_{n-1})\circ h_n =Q_0\circ \cdots \circ Q_{k-1}\circ h_k=\widetilde{f}^k.$$
The first inequality follows from the definition of $H$ and the second inequality follows from the assumption that $z$ is regular.
\end{proof}
Define
 \[\widetilde{\mathcal{K}}(n):=\{z\mid \tilde f^k(z) \notin A,\text{ for all } k \ge n\}\] and \[\widetilde{\mathcal{L}}(n)=\{z\mid \tilde f^k(z) \notin \mathcal B,\text{ for all } k \ge n\} \supset \widetilde{\mathcal{K}}(n).\]
 As $\mathcal{B}$ is open, $\tilde {\mathcal L}(n)$ is closed for each $n\ge 0$.

\begin{Lemma}For every $k>n$, $\overline{\partial } (h_k\circ h^{-1}_n)=0$ a.e. on $h_n(\widetilde{\mathcal{K}}(n))$. Moreover, there exists $K'$-q.c map $h_{k,n}$ such that $h_{k,n}=h_k\circ h^{-1}_n$ on $h_n(\widetilde{\mathcal{L}}(n))$.
\end{Lemma}
\begin{proof}

For each regular $z\in \widetilde{\mathcal{K}}(n)$, $\#\{n\le j<k: \widetilde{f}^j(z)\in A\}=0$. So by Lemma~\ref{lem:hkh-n1st}, $H(h_k\circ h_n^{-1}; h_n(z))=1$. Since a qc map is absolutely continuous, the first statement follows.

Now let us turn to the second statement. Note that for each regular $z\in \widetilde{\mathcal{L}}(n)$,
\begin{equation}\label{eqn:hkhn-1Ln}
H(h_k\circ h_n^{-1}; h_n(z))\le K^T,
\end{equation}
since by assumption ($\ast$), $\#\{n\le j<k:\widetilde{f}^j(z)\in A\}\le T.$ Put $K'=K^{4T+1}M$.

{\bf Claim.} For each component
$W$ of $\C\setminus \widetilde{\mathcal{L}}(n)$, there is a $K'$-qc map $h^W_{k,n}$ from $h_n(W)$ onto its image such that $h_{k,n}^W|\partial h_n(W)=h_k\circ h_n^{-1}|h_n(\partial W)$.

Once this claim is proved, we can obtain a homeomorphism $h_{k,n}:\CC\to \CC$ which coincides with $h_k\circ h_n^{-1}$ outside $\widetilde{\mathcal{L}}(n)$ and coincides with $h_{k,n}^W$ on $h_n(W)$ for each component $W$ of  $\C\setminus \widetilde{\mathcal{L}}(n)$. By \cite[Lemma 2 in Chapter 1]{DH2}, the map $h_{k,n}$ is $K'$-qc.

To prove the claim, let $w$ be the smallest integers such that $w\ge n$ and $\widetilde{f}^w(W)\cap\mathcal{B}\not=\emptyset$. Since $\mathcal{B}$ is nice and disjoint from the postcritical set of $\widetilde{f}$,
$\widetilde{f}^w$ maps $W$ homeomorphically onto a component $B$ of $\mathcal{B}$, and $\widetilde{f}^k(W)\cap\mathcal{B}=\emptyset$ for each $n\le k<w$.
By the assumption ($\ast$), for any $z\in W$,
$\#\{n\le j<w: \widetilde{f}^j(z)\in A\}\le T.$
By Lemma~\ref{lem:hkh-n1st}, if $z$ is regular, then for any $n<k\le w$,
$$H(h_k\circ h_n^{-1}; h_n(z))=H(\widetilde{f}^{k-n}; \widetilde{f}^n(z))\le K^T.$$
In particular, if $n<k\le w$, then $h_k\circ h_n^{-1}$ is $K^T$-qc on $h_n(W)$. Assume now that $k>w$ and let $W'=(\widetilde{f}^{w}|W)^{-1}(D(\mathcal{B})\cap B)$. Since $\widetilde{f}^{n}\circ h_n^{-1}$ is conformal on $h_n(W)$ and $\widetilde{f}^{w-n}$ is a $K^T$-qc map in $\widetilde{f}^n(W)$, by Lemma~\ref{qc1},
we have
$$\mathcal{QD}(h_n(W'), h_n(W))\le K^{2T} \mathcal{QD}(D(\mathcal{B})\cap B, B)\le K^{2T}M.$$
For each $z\in W\setminus W'$, $\widetilde{f}^w(z)\not\in D(\mathcal{B})$, so $\widetilde{f}^j(z)\not\in \mathcal{B}$ for all $j>w$. By assumption ($\ast$), it follows that
$$\#\{n\le j<k: \widetilde{f}^j(z)\in A\}\le 2T+1.$$
Thus for a regular $z\in W\setminus W'$, $H(h_k\circ h_n^{-1};h_n(z))\le K^{2T+1}.$
It follows by Lemma~\ref{qc2} that there is a $K'$-qc map defined on $h_n(W)$ which has the same boundary value as $h_k\circ h_n^{-1}$.
\end{proof}

By compactness of normalized $K$-qc maps and the diagonal argument, there exists a sequence $\{k_j\}_{j=1}^\infty$ of positive integers such that
for each $n$, $h_{k_j,n}$ converges locally uniformly in $\C$ to a $K'$-qc map
$\lambda_n:\C\to\C$. Since $\overline{\partial } h_{k_j,n}$ and $\overline{\partial }\lambda_n$ have bounded norm in $L^2(\C)$ and $\overline{\partial } h_{k_j,n}$ converges to $\overline{\partial }\lambda_n$ in the sense of distribution, it follows that
\begin{equation}\label{eqn:dbarchi}
\overline{\partial }\lambda_n=0 \text{ a.e. on } h_n(\widetilde{\mathcal{K}}(n)).
\end{equation}

\begin{Lemma}\label{lem:chinhn}
For each $k>n\ge 0$,
$$\lambda_n \circ h_n=\lambda_k \circ h_k\text{ on }\widetilde{\mathcal{L}}(n).$$
\end{Lemma}
\begin{proof} For each $j$ large enough so that $k_j>k>n$, since $\widetilde{\mathcal{L}}(n)\subset \widetilde{\mathcal{L}}(k)$,
\begin{eqnarray*}
h_{k_j,n}\circ h_n=h_{k_j} = h_{k_j}\circ h^{-1}_k\circ h_k
= h_{k_j,k}\circ h_k
\end{eqnarray*}
is valid on $\widetilde{\mathcal{L}}(n)$. Letting $j$ go to infinity, we obtain $\lambda_n\circ h_n=\lambda_k\circ h_k$ on $\widetilde{\mathcal{L}}(n)$.
\end{proof}

\subsection{Limit geometry}
Let $\mathcal{L}(n)=\lambda_n\circ h_n(\widetilde{\mathcal{L}}(n))$ and $\mathcal{K}(n)=\lambda_n\circ h_n(\widetilde{\mathcal{K}}(n))$. Since we assume $A\supset \mathcal{B}$, $\mathcal{K}(n)\subset \mathcal{L}(n)$. By Lemma~\ref{lem:chinhn}, for $k\ge n$, we have
\[\mathcal{L}(n)=\lambda_k\circ h_k(\widetilde{\mathcal{L}}(n))\subset \lambda_k\circ h_k(\widetilde{\mathcal{L}}(k))=\mathcal{L}(k)\]
and \[\mathcal{K}(n)=\lambda_k\circ h_k(\widetilde{\mathcal{K}}(n))\subset \lambda_k\circ h_k(\widetilde{\mathcal{K}}(k))=\mathcal{K}(k).\]
By (\ref{eqn:dbarchi}), $\lambda^{-1}_n$ is conformal a.e.on $\mathcal{K}(n)$.

\begin{equation*}
\xymatrix@C=4em@R=3em{
\cdots \ar[r]  &\mathbb C\ar[r] &\mathbb C \ar[r] &\cdots \ar[r] &\mathbb C\ar[r] &\mathbb C\\
\cdots \ar[r]  &\mathbb C  \ar[u]^{\lambda_{n+1}}\ar[r]^{Q_n} &\mathbb C\ar[u]^{\lambda_{n}}\ar[r] &\cdots \ar[r] &\mathbb C \ar[u]^{\lambda_{1}}\ar[r]^{Q_0}&\mathbb C\ar[u]^{\lambda_{0}}\\
\cdots \ar[r] &\mathbb C \ar[u]^{h_{n+1}}\ar[r]^{\tilde f} &\mathbb C\ar[u]^{h_n}\ar[r] &\cdots \ar[r] &\mathbb C
\ar[u]^{h_1}\ar[r]^{\tilde f} &\mathbb C \ar[u]^{\mathrm{id}}}
\end{equation*}

\begin{Lemma}\label{lem:shapeW} There exists $M'>0$ such that for any $n\ge 0$ and any component $W$ of $\CC\setminus \mathcal{L}(n)$,
$$\text{diam} (W)^2\le M'\text{area}(W).$$
\end{Lemma}
\begin{proof} Let $\widetilde{W}=(\lambda_n\circ h_n)^{-1}(W)$. Let $w$ be the minimal integer such that $w\ge n$ and $\widetilde{f}^w(\widetilde{W})\cap \mathcal{B}\not=\emptyset$. Then $\widetilde{W}$ is a component of $\CC\setminus \widetilde{\mathcal{L}}(w)$ and $\widetilde{f}^w$ maps $\widetilde{W}$ homeomorphically onto a component $B$ of $\mathcal{B}$. So $\varphi:=Q_0\circ Q_1\circ \cdots Q_{w-1}$ maps $h_w(\widetilde{W})$ conformally onto $B$. Moreover, since $B$ has a definite neighborhood disjoint from $P(\widetilde{f})$, the conformal map
$\varphi$ has bounded distortion. Thus $\text{diam}(h_w(\widetilde{W}))^2/\text{area}(h_w(\widetilde{W}))$ is bounded from above. Since $\lambda_w$ are normalized $K'$-qc maps and $\lambda_w (h_w(\widetilde{W}))=W$, the statement follows.
\end{proof}

\begin{Lemma}\label{lem:Lnsmall}
\begin{enumerate}
\item The Lebesgue measure of the set $\C\setminus \mathcal{L}(n)$ tends to zero as $n\to\infty$.
\item $$\lim_{n\to\infty} \sup\{\text{diam}(W): W\text{ is a component of }\C\setminus \mathcal{L}(n)\}=0.$$
\end{enumerate}
\end{Lemma}
\begin{proof}
The second statement follows form the first since all components of $\CC\setminus \mathcal{L}(n)$ have uniformly bounded shape by Lemma~\ref{lem:shapeW}.

To prove the first statement, we shall use a martingale type argument. Let $\mathscr{W}$ denote the collection of components of $\CC\setminus \mathcal{L}(n)$, where $n$ runs over all non-negative integers. Let $\mathscr{W}^0$ denote the maximal elements in $\mathscr{W}$, i.e., those that are not contained in any other. For each $k\ge 1$, define inductively $\mathscr{W}^k$ to be the maximal elements in $\mathscr{W}\setminus \bigcup_{0\le j<k} \mathscr{W}^j$. So $\mathscr{W}$ is the disjoint union of $\mathscr{W}^k$, $k=0,1,\ldots$. Note that
$$\CC\setminus \bigcup_n \mathcal{L}(n)\subset \bigcap_{k=0}^\infty \bigcup_{W\in \mathscr{W}^k} W.$$
It suffices to show that there is a constant $\lambda\in (0,1)$ such that for $k\ge 0$ and each $W\in \mathscr{W}^k$,
\begin{equation}\label{eqn:Wfollower}
\frac{\text{area}(W\cap (\bigcup_{W'\in \mathscr{W}^{k+1}}W'))}{\text{area} (W)}\le \lambda.
\end{equation}

To this end, fix such a $W$. Note that there is $n$ such that $\widetilde{W}:=(\lambda_n\circ h_n)^{-1}(W)$ satisfies the following: $\widetilde{f}^n$ maps $\widetilde{W}$ homeomorphically onto a component $B$ of $\mathcal{B}$.
So $\varphi:=Q_0\circ Q_1\circ \cdots\circ Q_{n-1}$ maps $\lambda_n^{-1}(W)$ conformally onto $B$ and maps $\lambda_n^{-1}(W\setminus (\bigcup_{W'\in \mathscr{W}^{k+1}}W'))$ onto  $B\setminus D(\mathcal{B})$.
Since $B$ has a definite neighborhood disjoint from $P(\widetilde{f})$, $\varphi|\lambda_n^{-1}(W)$ has bounded distortion. Thus
$$\frac{\text{area}(\lambda_n^{-1}(W\setminus (\bigcup_{W'\in \mathscr{W}^{k+1}}W'))}{\text{area} (\lambda_n^{-1}(W))} \ge C\frac{\text{area}(B\setminus D(\mathcal{B}))}{\text{area}(B)}>\frac{C}{M},$$
where $C>0$ is independent of $W$.
Since $\lambda_n$ are normalized $K'$-qc maps, (\ref{eqn:Wfollower}) follows.
\end{proof}

\begin{Lemma} \label{lem:KnLn}
$\bigcup_{n=0}^\infty \mathcal{K}(n)=\bigcup_{n=0}^\infty\mathcal{L}(n)$.
\end{Lemma}
\begin{proof} By the assumption ($\ast$), $\bigcup_n \widetilde{\mathcal{K}}(n)=\bigcup_n \widetilde{\mathcal{L}}(n)$. Arguing by contradiction, assume that there exists $z\in (\bigcup_n \mathcal{L}(n))\setminus (\bigcup_n \mathcal{K}(n))$. Then there exists $n_0$ such that for all $n\ge n_0,$ $z\in \mathcal{L}(n)\setminus \mathcal{K}(n)$. By definition, there exists $\tilde{z}_n\in \widetilde{\mathcal L}(n)\setminus \widetilde{\mathcal{K}}(n)$ such that $\lambda_n\circ h_n(\tilde{z}_n)=z$. Then $$\lambda_{n+1}\circ h_{n+1}(\tilde{z}_n)=\lambda_n\circ h_n(\tilde{z}_n)=z=\lambda_{n+1}\circ h_{n+1}(\tilde{z}_{n+1}),$$ so $\tilde{z}_{n+1}=\tilde{z}_n$. Therefore $\tilde{z}=\tilde{z}_{n_0}$ satisfies
$\tilde{z}\in (\bigcup_{n=n_0}^\infty \widetilde{\mathcal{L}}(n))\setminus (\bigcup_{n=0}^\infty\widetilde{\mathcal{K}}(n))$. This is absurd.
\end{proof}
\begin{proof}[Proof of Theorem~\ref{thm:Thurston}]
By Lemma~\ref{lem:chinhn}, $\lambda_n\circ h_n=\lambda_k\circ h_k$ on $\widetilde{L}(n)$ for all $k>n$, we obtain $\lambda_n\circ h_n(\widetilde{W})=\lambda_k\circ h_k(\widetilde{W})$ for all component $\widetilde{W}$ of $\mathbb C\backslash \widetilde{L}(n)$. By Lemma~\ref{lem:Lnsmall},
\[\sup\limits_{z\in \mathbb C}|\lambda_k\circ h_k(z)-\lambda_n\circ h_n(z)| \le 2\sup\limits_{\widetilde W}\mathrm{diam}~\lambda_n\circ h_n(\widetilde W) \to 0\]
as $n$ tends to $\infty$, so $\lambda_n\circ h_n$ converges uniformlly to a continuous function $h:\mathbb C\to \mathbb C$. Similarly,
\begin{eqnarray*}
& &\sup\limits_{z\in \mathbb C}|\lambda_{k}\circ Q_{k}\circ \lambda^{-1}_{k+1}(z)-\lambda_{n}\circ Q_{n}\circ \lambda^{-1}_{n+1}(z)|\\
 &=& \sup\limits_{z\in \mathbb C}|\lambda_{k}\circ h_k\circ\tilde f\circ h^{-1}_{k+1}\circ \lambda^{-1}_{k+1}(z)-\lambda_{n}\circ h_n\circ \tilde f \circ h^{-1}_{n+1}\circ \lambda^{-1}_{n+1}(z)| \\
 &=& 2 \sup\limits_{W}\mathrm{diam}W \to 0
\end{eqnarray*}
as $n$ tends to $0$, so $\lambda_{n}\circ Q_{n}\circ \lambda^{-1}_{n+1}$ converges uniformlly to a proper map $f:\mathbb C\to \mathbb C$.\par
Recall that $\lambda^{-1}_n$ is a normalized $\tilde K$-qc map and conformal Lebsgue a.e. on $\mathcal{K}(n)$. By Lemma~\ref{lem:Lnsmall} (1) and Lemma~\ref{lem:KnLn}, $\bigcup\limits_{n}\mathcal{K}(n)=\bigcup\limits_n \mathcal{L}(n)$ has full Lebsgue measure. By \cite[Lemma B.1]{R-L}, $\lambda^{-1}_n$ converges uniformly to identity, hence so does $\lambda_n$.

We conclude that $h_n$ converges uniformly to a continuous map $h$ and $Q_n$ converges uniformly to a rational map $f$ of degree $d$.
Thus $h\circ \widetilde{f}=f\circ h$.
On $\widetilde{\mathcal{L}}(0)$, $h(z)=\lim_{n\to\infty}\lambda_n\circ h_n(z)=\lambda_0\circ h_0(z)=\lambda_0(z)$.
\end{proof}

\section{Qc surgery and proof of the Main Theorem}\label{sec:surgery}
As before, we fix $f_0\in \Poly(d)$ which is postcritically finite, hyperbolic and primitive, let $T=(|T|, \sigma, \delta)$ denote the reduced mapping scheme of $f_0$ and let $r:|T|\to \mathbb{N}$ denote the return time function. We also fix a collection $\{\theta_\v\}_{\v\in |T|}$ of external angles such that $d^{r(\v)}\theta_\v=\theta_{\sigma(\v)}\mod 1$ and such that $\mathcal{R}_{f_0}(\theta_\v)$ lands on the boundary of $\v$, for each $\v\in |T|$, according to Lemma~\ref{lem:externalangle}.

Choosing two large positive integers $N_0<N_1$ such that the following hold for each $\v\in |T|$:
\begin{itemize}
\item $f_0^j(Y_{N_0}(\v))$, $\v\in |T|$, $0\le j<r(\v)$, are pairwise disjoint;
\item $f_0^{r(\v)}: Y_{N_0}(\v)\to Y_{N_0-r(\v)}(\sigma(\v))$ has degree $\delta(\v)$;
\item putting $N'_0=N_0+N+\max_{\v} r(\v)$, $Y_{N_1}(\v)\Subset Y_{N'_0}(\v)$.
\end{itemize}
where $N$ is as in Theorem~\ref{thm:Kqcpuzzle}.

Applying the `thickening' procedure (\cite{Mil4} and \cite[Lemma 5.13]{IK}),
we obtain quasi-disks $U'_{\v}\Subset U''_{\v}$, with $U'_{\v}\supset Y_{N_1+r(\v)}(\v)$ and $Y_{N_0+N}(\v)\Supset U''_{\v}\supset Y_{N_1}(\v)$ and such that $f_0^{r(\v)}: U'_{\v}\to U''_{\v}$ again has degree $\delta(\v)$. Then the map
$$F_0:\bigcup_{\v\in |T|} \{\v\}\times U'_\v\to \bigcup_{\v} \{\v\}\times U''_\v$$
defined by $F_0|U'_\v=f_0^{r(\v)} |U'_\v$, is a GPL map over $T$, with filled Julia set equal to $\bigcup_{\v}\{\v\}\times\overline{\v}$.  We may choose these domains $U'_\v$ and $U''_\v$ so that
$\mathcal{R}_{f_0}(\theta_\v)$ intersects $\partial U'_\v$ (resp. $\partial U''_\v$) at a single point.
Let $U_\v=\{z\in U'_\v: F_0(\v, z)\in \{\sigma(\v)\}\times U'_{\sigma(\v)}\}$.

\begin{Theorem}\label{thm:tildef} Given $g\in\mathcal{C}(T)$ there exists a quasi-regular map $\widetilde{f}$ of degree $d$ with the following properties:
\begin{enumerate}
\item $f_0(z)=\widetilde{f}(z)$ for each $z\in \C\setminus (\bigcup_{\v} U'_\v)$;
\item There exist quasi-disks $U_{\v, g} \Subset U'_\v$ such that $\widetilde{f}$ is holomorphic in $U_{\v,g}$ and
the map $$\widetilde{F}:\bigcup_{\v} \{\v\}\times U_{\v, g} \to \bigcup_{\v} \{\v\}\times U'_\v,\,\,
(\v, z)\mapsto (\sigma(\v), \widetilde{f}^{r(\v)}(z)),$$ is a GPL map over $T$ which is conformally conjugate to $g$ near their filled Julia set.
More precisely, there are  quasi-disks $V_{\v, g}\Subset V'_{\v, g}$ such that
$g:\bigcup_{\v} \{\v\}\times V_{\v, g} \to \bigcup_{\v} \{\v\}\times V'_{\v,g}$ is a GPL map over $T$, and
for each $\v\in |T|$ there is a conformal map $\varphi_\v: U'_\v\to V'_{\v,g}$ such that $\varphi_{\v} (U_{\v, g})=V_{\v,g}$ and
$$\varphi_{\sigma(\v)}\circ \widetilde{f}^{r(\v)}=g\circ \varphi_\v \text{ holds on } U_{\v,g}.$$
\item Furthermore, if $\ell_\v$ denote the union of $\mathcal{R}_{f_0}(\theta_\v)\setminus U'_\v$ and $\varphi_\v^{-1}(\mathcal{R}_g(\v, 0)\cap V'_{\v,g})$, then $\ell_\v$ is a ray, that is a simple curve starting from the infinity,
        and $\widetilde{f}^{r(\v)}(\ell_\v)=\ell_{\sigma(\v)}$.
\end{enumerate}
\end{Theorem}

\begin{proof}
Let $a'_\v$ (resp. $a_\v$) denote the unique intersection point of $\mathcal{R}_{f_0}(\theta_\v)$ with $\partial U'_\v$ (resp. $U_\v$).
Let $V'_{\v,g}=\{z\in \C: |G_g(\v,z)|<1\}$ and $V_{\v,g}=\{z\in \C: |G_g(\v,z)|<1/\delta(\v)\}$, where $G_g$ is the Green function of $g$. Let $a'_{\v,g}$ (resp. $a_{\v, g}$) denote the unique intersection point of the external ray $\mathcal{R}_g(\v, 0)$ with $\partial V'_{\v, g}$ (resp. $\partial V_{\v, g}$).

Let $\varphi_\v$ denote the unique Riemann mapping from $U'_\v$ onto $V'_{\v,g}$ such that $\varphi_\v(a'_{\v})=a'_{\v, g}$ and $\varphi_\v (a(\v))=a_{\v, g}$. Define $U_{\v,g}=\varphi_{\v}^{-1}(V_{\v,g})$ and define
$\widetilde{f}|U_{\v,g}= (f^{r(\v)-1}|f(U'_\v))^{-1}\circ (\varphi_{\sigma(\v)}^{-1}\circ g\circ \varphi_\v)$ which is a holomorphic proper map of degree $\delta(\v)$. Finally, define $\widetilde{f}$ on each annulus $U'_{\v}\setminus U_{\v, g}$ so that $\widetilde{f}$ is a quasiregular covering map from this annulus to $f_0(U'_{\v}\setminus U_{\v})$ of degree $\delta(\v)$ and $\widetilde{f}$ maps the arc $\varphi^{-1}_{\v}(\mathcal{R}_g(\v,0)\cap (V_\v'\setminus V_\v))$ onto the arc $f_0(\mathcal{R}_{f_0}(\theta_\v)\cap (U_\v'\setminus U_\v))$. All the desired properties are easily checked.
\end{proof}

\begin{proof}[Proof of the Main Theorem (Surjectivity)] Let $\widetilde f$ be the quasi-regular map constructed in Theorem~\ref{thm:tildef}, and let $C\ge 1$ be the maximal dilatation of $\widetilde{f}$.
Let us check that it satisfies the conditions of Theorem~\ref{thm:Thurston}. Firstly, it is a quasi-regular polynomial and $\bar{\partial} f=0$ a.e. on $\C\setminus A$, where
$$A:=\bigcup\limits_{\v \in |T|} U'_{\v}\backslash U_{\v,g}\subset \bigcup_{\v} Y_{N_0+N}(\v).$$  To construct the set $\mathcal{B}$, let
$$R(A)=\{x\in A:\exists n\ge 1\text{ such that } \widetilde{f}^n(z)\in A\}$$ be the return domain to $A$ under $\widetilde{f}$.
For every $x \in R(A)$, there is a smallest integer $k=k(x)$ such that $\tilde f^k(x) \in \bigcup\limits_{\v} Y_{N_0}(\v)\backslash \overline{Y_{N_0+r(\v)}(\v)}=:\Omega$. It is easy to see $Q:=\sup\limits_{x \in R(A)}k(x)<\infty$.  Let
$$E=\left\{z\in \Omega: \exists n\ge 1, \text{ such that } \widetilde{f}^n(z)\in \bigcup_{\v}Y_{N_0}(\v)\right\},$$
and let $\mathcal{B}$ be the union of components  of $D(E)$ which intersect $R(A)$. So $\mathcal{B}\supset R(A)$.
Consequently, if $\widetilde{f}^j(z)\not\in \mathcal{B}$ for $0\le j<n$, then $\#\{0\le j<n: \widetilde{f}^j(z)\in A\}\le 1$. So the condition ($\ast$) in Theorem~\ref{thm:Thurston} holds with $T=1$.

Let us show that $\mathcal{B}$ is a free $M$-nice set for some $M>0$. To this end, we first observe that $E$ and hence $\mathcal{B}$ is a nice set of $\widetilde{f}$.  Since $\overline{\Omega}$ is disjoint from the post-critical set of $\widetilde{f}$ and $\mathcal{B}\subset \bigcup_{k=0}^Q \widetilde{f}^{-k}(\Omega)$, $\mathcal{B}$ is free.
Fix a component $B$ of $\mathcal{B}$, let $s$ be the entry time of $B$ into $E$, and let $t$ denote the return time of $\widetilde{f}^{s}(B)$ into $\bigcup_{\v} Y_{N_0}(\v)$. Then $\widetilde{f}^{s+t}$ maps $B$ homeomorphically onto a component of $\bigcup_{\v}Y_{N_0}(\v)$
and the map $\widetilde{f}^{s+t}|B$ is $C$-qc. In fact, for each $x\in B$, $\#\{0\le j<s: \widetilde{f}^j(x)\in A\}\le 1$ and $\widetilde{f}^t|\widetilde{f}^s(B)=f_0^t|\widetilde{f}^s(B)$ is conformal.
Our assumption on $N_1$ and $N_0$ ensures that $B\subset \mathcal{B}\subset \bigcup_{\v}Y_{N_0+N}(\v)$, so that
$\widetilde{f}^{s+t} (D(\mathcal{B})\cap B)\subset L_{N_0+N}$. It follows from Lemma~\ref{qc1} and Theorem~3'
that
\[\mathcal{QD}(D(\mathcal {B})\cap B, B) \le C^{2}\max_{\v\in |T|}(L_{N_0+N}\cap Y_{N_0}(\v),Y_{N_0}(\v))=:M<\infty.\]
Since $\widetilde{f}^{s+t}|B$ extends to a $C$-qc map onto a neighborhood of $\widetilde{f}^{s+t}(B)$, enlarging $M$ if necessary, we have that
$M\text{area}(B)\ge \text{diam}(B)^2$ and $M\text{area}(B\cap D(\mathcal{B}))>\text{area}(B)$.
This proves that $\mathcal{B}$ is $M$-nice.

So by Theorem~\ref{thm:Thurston}, there is a continuous surjective map $h$ and a map $f\in Poly_d$
such that $f\circ h=h\circ \widetilde{f}$. The map $h$ is holomorphic and $h(z)=z+o(1)$ near infinity. Near $\infty$, $\widetilde{f}=f_0$. Thus $h=\phi_f\circ \phi_{f_0}^{-1}$ near $\infty$, where $\phi_f$ and $\phi_{f_0}$ are the B\"ottcher map for $f$ and $f_0$ respectively. It follows that $h(\ell_\v)$ is the external ray $\mathcal{R}_f(\theta_\v)$. By Proposition~\ref{prop:com}, $f\in \mathcal{C}(f_0)$ and $F:\bigcup_{\v}\{\v\}\times h(Y_{N_0+r(\v)}(\v))\to \bigcup_{\v}\{\v\}\times h(Y_{N_0}(\v))$, $F|h(Y_{N_0+r(\v)}(\v))=f^{r(\v)}$, is a $\lambda(f_0)$-renormalization of $f$.

In order to show that $\chi(f)=g$, we need to show that $F$ and $\widetilde{F}$ are hybrid equivalent. Let us consider the associated maps $\widetilde{\textbf{F}}: \bigcup_{\v} Y_{N_0+r(\v)}(\v)\to \bigcup_{\v} Y_{N_0}(\v)$ and
$\textbf{F}:\bigcup_{\v} h(Y_{N_0+r(\v)}(\v))\to \bigcup_{\v}h(Y_{N_0}(\v))$, where $\widetilde{\textbf{F}}|Y_{N_0+r(\v)}(v)=\widetilde{f}^{r(\v)}|Y_{N_0+r(\v)}(\v)$ and $\textbf{F}|h(Y_{N_0+r(\v)})=f^{r(\v)}|h(Y_{N_0+r(\v)}(\v))$. It suffices to show that there is a qc map $H:\bigcup_{\v} U'_\v\to \bigcup_{\v} h(U'_\v)$ such that $H\circ \widetilde{\textbf{F}}=\textbf{F}\circ H$ and such that $\bar{H}=0$ a.e. on the filled Julia set $K(\widetilde{\textbf{F}})$ of $\widetilde{\textbf{F}}$. Note that $h\circ \widetilde{\textbf{F}}= \textbf{F}\circ h$ and $K(\widetilde{\textbf{F}})$ contains the postcritical set of $\widetilde{\textbf{F}}$.
By Theorem~\ref{thm:Thurston}, there is a qc map $\lambda_0$ such that $\lambda_0=h$ outside $\mathcal{B}':=\bigcup_{k=0}^\infty \widetilde{f}^{-k}(\mathcal{B})$ and such that $\bar{\partial } \lambda_0=0$ a.e. on $\{z: \widetilde{f}^n(z)\notin A, \forall n\ge 0\}$. In particular, $\bar{\partial} \lambda_0=0$ a.e. on $K(\widetilde{\textbf{F}})$. Since $\mathcal{B}'$ is a countable union of Jordan disks with pairwise disjoint closure and it is disjoint from $K(\widetilde{\textbf{F}})$, $\lambda_0$ is homotopic to $h$ rel $K(\widetilde{\textbf{F}})$.
Therefore, there is a sequence of qc maps $\lambda_k:\bigcup_{\v} Y_{N_0}(\v)\to \bigcup_{\v} h(Y_{N_0}(\v))$, $k=1,2,\ldots$, all homotopic to $h$ rel $K(\widetilde{\textbf{F}})$, such that $\textbf{F}\circ \lambda_{k+1}=\lambda_k \circ \widetilde{\textbf{F}}$ and $\lambda_k=\lambda_0$ on $W:=\bigcup_{\v} (Y_{N_0}(\v)\setminus Y_{N_0+r(\v)}(\v))$ holds for $k=0,1,\ldots$. Since $\textbf{F}$ is holomorphic, $\widetilde{\textbf{F}}$ is holomorphic outside $A$,
and each orbit of $\widetilde{\textbf{F}}$ passes through $A$ at most once, the maximal dilatation of $\lambda_k$ is uniformly bounded. Since $\lambda_k(z)$ eventually stablizes for each $z$ in the domain of $\widetilde{\textbf{F}}$, $\lambda_k$ converges to a qc map $H$. Moreover, $\bar{\partial }H =0$ holds a.e. on $K(\widetilde{\textbf{F}})$ since so does $\lambda_k$ for each $k$.
\end{proof}

In order to show that $\mathcal{C}(f_0)$ is connected, we shall make use of the following result.
\begin{Theorem}[Branner-Hubbard-Lavaurs]\label{thm:bhl}
The set $\mathcal{C}(T)$ is a connected compact set.
\end{Theorem}
\begin{proof} The proofs in the literature were stated for the case $\mathcal{C}(d)$. For $d=2$ this is due to Douady-Hubbard (\cite{DH1}), the case $d=3$ was proved by Branner-Hubbard (\cite{BH}) and for all $d\ge 3$ this was proved by Lavaurs (\cite{La}).
The proof of Lavaurs generalizes to the case of $\mathcal{C}(T)$ in a straightforward way. See also \cite{DP} for a stronger result with a different proof.
\end{proof}
\begin{proof}[Proof of the Main Theorem (connectivity)] We shall show that if $E$ is a non-empty open and closed subset of $\mathcal{C}(f_0)$, then $\chi(E)$ is a closed subset of $\mathcal{C}(T)$. Together with connectivity of $\mathcal{C}(T)$ and bijectivity of the map $\chi$, this implies that $\mathcal{C}(f_0)$ is connected.

Suppose that $g_n$ is a sequence in $\chi(E)$ and $g_n\to g$ in $\mathcal{C}(T)$. We need to show that $g\in \chi(E)$. Let $f_n=\chi^{-1}(g_n)\in E$. Since $E$ is compact, passing to a subsequence we may assume that $f_n\to f\in E$. As in \cite[Section 7]{DH2}, we may choose hybrid conjugacies $h_n$ between $\lambda(f_0)$-renormalization of $f_n$ and $g_n$ so that the maximal dilatation of $h_n$ is uniformly bounded. Passing to a further subsequence, we see that the $\lambda(f_0)$-renormalization of $f$
is qc conjugate to $g$ respecting the external markings. Thus $f$ is conjugate to $\chi^{-1}(g)$ via a qc map $h:\C\to \C$ which is conformal outside the filled Julia set of $f$ and satisfies $h(z)=z+o(1)$ near infinity.  The Beltrami path connecting $f$ and $\chi^{-1}(g)$ is contained in $E$ and thus $g=\chi(\chi^{-1}(g))\in \chi(E)$.
\end{proof}

\bibliographystyle{plain}             

\begin{thebibliography}{1}

\bibitem{BH} B. Branner, J. Hubbard. {\em The iteration of cubic polynomials Part I: The global topology of parameter space. } Acta Math., {\bf 160} (1988), 143-206.

\bibitem{DP} L. De Marco and K. Pilgrim. {\em Polynomial basins of infinity.} Geom. Funct. Anal. {\bf 21} (2011), 920-950.
\bibitem{EY} A. Epstein and M. Yampolsky.  {\em Geography of the cubic connectedness locus: intertwining surgery}. In Ann. Sci. Ecole Norm. Sup. {\bf 32}, 151-185.
\bibitem{DH1} A. Douady and H. Hubbard, {\em \'Etude dynamique des polyn\^omes complexes.} Publications Math\'ematiques d'Orsay, Universit\'e de Paris-Sud, D\'epartement de Mathematiques, Orsay, 1984
\bibitem{DH2} A. Douady and H. Hubbard, {\em On the dynamics of polynomial-like mappings}, Ann. Sci. Ecole Norm. Sup. (4) {\bf 18}
(1985) 287-343.
\bibitem{DH3} A. Douady and H. Hubbard. {\em A proof of Thurston's topological characterization of rational functions.} Acta Math., 171(2), 263-297.
\bibitem{I1} H. Inou, {\em Renormalization and rigidity of polynomials of higher degree}, J. Math. Kyoto Univ. {\bf 42} (2002) 351-392.
%
\bibitem{I3} H. Inou, {\em Combinatorics and topology of straightening maps II: discontinuity}, arXiv:0903.4289v2.
\bibitem{IK} H. Inou, J. Kiwi, {\em Combinatorics and topology of straightening maps I: compactness and bijectivity}. Adv. Math. {\bf 231} (2012), 2666-2733
\bibitem{KSS} O. Kozlovski, W. Shen, S. van Strien. {\em Rigidity for real polynomials}, Ann. Math., {\bf 165} (2007), 749-841

\bibitem{LY} M. Lyubich and M. Yampolsky. {\em Dynamics of quadratic polynomials: complex bounds for real maps.} Ann. Inst. Fourier {\bf 47} (1997), 1219-1255.
\bibitem{R-L} J. Rivera-Letelier, {\em A connecting lemma for rational maps satisfying a no growth condition}. Ergodic Theory Dynam. System {\bf 27} (2007),no.2, 595-636.
\bibitem{Kahn} J. Kahn, {\em Holomorphic Removability of Julia Sets}. arXiv:math/9812164

\bibitem{La} P. Lavaurs. PhD Thesis at University of Paris Sud, Orsay, 1989.


\bibitem{McM1} C. McMullen, {\em Complex dynamics and renormalization}, Annals of Mathematics Studies, vol. 135, Princeton University Press, Princeton, NJ, 1994
\bibitem{Mil1} J. Milnor, {\em Dynamics in one complex variable. Third Edition}, Annals of Mathematics Studies {\bf 160}, Princeton University Press, 2006.
\bibitem{Mil3}J. Milnor, {\em Hyperbolic components}, in: Conformal Dynamics and Hyperbolic Geometry, in: Contemp. Math., vol. {\bf 573}, Amer. Math. Soc., Providence, RI, 2012, pp. 183-232.
\bibitem{Mil4}J. Milnor, {\em Local connectivity of Julia sets: expository lectures.}  The Mandelbrot Set, Theme and Variations, volume {\bf 274} of London Math. Soc. Lecture Note Ser., 67-116. Cambridge Univ. Press, Cambridge.
\bibitem{RY} P.Roesch, Y. Yin.{\em The boundary of bounded polynomial Fatou components.} C. R. Math. Acad. Sci. Paris, {\bf 346}, 877-880.
\bibitem{Shishi} M.Shishikura, {\em On the quasiconformal surgery of rational functions.} Ann. Sci. Ecole Norm. Sup. {\bf 20} (1987),  1-29.
\end{thebibliography}

\end{document}